\documentclass[11pt,oneside]{article}

\usepackage{latexsym}
\usepackage{shortvrb}
\usepackage{graphicx}
\usepackage{fancyhdr}
\usepackage{amsfonts}
\usepackage{amsmath}
\usepackage{amssymb}
\usepackage{mathrsfs}
\usepackage{makeidx}
\usepackage[all]{xy}
\usepackage{amsthm}
\usepackage{tikz}
\usetikzlibrary{matrix,arrows,decorations.pathmorphing}

\usepackage[T1]{fontenc}
\usepackage{amsfonts}
\usepackage{amsmath}

\newtheorem{thm}{Theorem}[section]
\newtheorem{lem}[thm]{Lemma}
\newtheorem{prop}[thm]{Proposition}

\newtheorem{cor}[thm]{Corollary}

\theoremstyle{definition}
\newtheorem{definition}[thm]{Definition}

\newcommand{\blackged}{\hfill$\blacksquare$}
\newcommand{\whiteged}{\hfill$\square$}
\newcounter{proofcount}

\renewenvironment{proof}[1][\proofname.]{\par
  \ifnum \theproofcount>0 \pushQED{\whiteged} \else \pushQED{\blackged} \fi%
  \refstepcounter{proofcount}
  \normalfont 
  \trivlist
  \item[\hskip\labelsep
        \itshape
    {\bf\em #1}]\ignorespaces
}{%
  \addtocounter{proofcount}{-1}
  \popQED\endtrivlist
}

\begin{document}

\begin{center}
\textbf{Internalizing decorated bicategories: The globularily generated condition}
\end{center}

\begin{center}
\small\textit{Juan Orendain}
\end{center}

\small {\noindent \textit{Abstract:} This is the first part of a series of papers studying the problem of existence of double categories for which horizontal bicategory and object category are given. We refer to this problem as the problem of existence of internalizations for decorated bicategories. We establish a formal framework within which the problem of existence of internalizations can be correctly formulated. Further, we introduce the condition of a double category being globularily generated. We prove that the problem of existence of internalizations for a decorated bicategory admits a solution if and only if it admits a globularily generated solution, and we prove that the condition of a double category being globularily generated is precisely the condition of a solution to the problem of existence of internalizations for a decorated bicategory being minimal. The study of the condition of a double category being globularily generated will thus be pivotal in our study of the problem of existence of internalizations.} 
\normalsize

\tableofcontents

\section{Introduction}

\noindent There exists, in the mathematical literature, a variety of competing definitions of what a higher order categorical structure should be [3],[18]. Most common amongst which are those defined by the concepts of internalization and enrichment. These two ideas reduce, in the case of categorical structures of order 2, to the concepts of double category and bicategory respectively, both types of structures introduced by Ehresmann, in [10] and [11]. These two notions are related in different ways.
Every double category admits an underlying bicategory, its horizontal bicategory [22], and every bicategory can be considered as a trivial double category. Both constructions admit first order categorical extensions. The condition of a double category being trivial can be seen as a minimizing constraint with respect to the horizontal bicategory functor. 

We are interested in the existence of operations inverse to that of horizontalization. Precisely, we are interested in the problem of existence of double categories having a given bicategory as horizontal bicategory and a given category as category of objects. We call this problem the problem of existence of internalizations for decorated bicategories. We understand the problem of existence of internalizations of decorated bicategories as the problem of coherent equivariantization of 2-cells in bicategories, with respect to given collections of vertical morphisms. Our main motivation for the study of the problem of internalization of decorated bicategories comes from the work of Bartels, Douglas, and Henriques on the theory of correspondences between von Neumann algebras [4], from their theory of coordinate free conformal nets [5],[6],[7], and from the theory of regular and extended quantum field theories [1],[2],[19],[21].

This is the first installment of a series of papers studying the problem of existence of internalizations of decorated bicategories. We introduce the concepts of decorated bicategory and decorated pseudofunctor, which we use as formal framework within which to state the problem of existence of internalizations. Further, we introduce globularily generated double categories. We prove that the problem of existence of internalizations of decorated bicategories is equivalent to the problem of existence of globularily generated internalizations, and we prove that the condition of a double category being globularily generated is precisely the condition of a solution to the problem of existence of internalizations being minimal. This will be our main motivation for the study of globularily generated double categories. We interpret the condition of a double category being globularily generated as a decorated analog of the condition of a double category being trivial. We introduce technical tools necessary to obtain results on the theory of globularily generated double categories, of which we make use in order to prove that the globularily generated condition is not trivial. We present categorifications of structures involved in the theory of globularily generated double categories and we perform relevant computations in certain examples. The condition of a double category being globularily generated will be our main object of study. We now sketch the contents of this paper.

In section 2 we recall basic concepts related to the theory of categorical structures of second order, we introduce the concepts of decorated bicategory, decorated pseudofunctor, and decorated horizontalization; we introduce both the problem of existence of internalizations of decorated bicategories and the problem of existence of internalization functors. We establish notational conventions used throughout the paper and present examples relevant to our discussion. In section 3, motivated by problems presented in section 2 we define and study the concept of globularily generated double category. We define the globularily generated piece construction, which we furnish with the structure of a 2-functor. We use this to prove that the problem of existence of internalizations of decorated bicategories is equivalent to the problem of existence of globularily generated bicategories, and that the problem of existence of internalization functors is equivalent to the problem of existence of globularily generated internalization functors. Further, we prove that globularily generated double categories are precisely solutions to problems of existence of internalizations, which are minimal. We prove that the globularily generated piece 2-functor is a strict 2-reflector which we interpret by saying that the condition of a double category being globularily generated is a decorated analog of the condition of a double category being trivial. In section 4 we introduce the technical framework needed in order to obtain results on the structure of globularily generated double categories. We introduce the vertical filtration of a globularily generated double category and the vertical length of a globularily generated 2-morphism. We use this technical framework to prove that the condition of a double category being globularily generated is not trivial. In section 5 we extend the definition of vertical filtration of a globularily generated double category introduced in section 4, to a filtration of the globularily generated piece functor. In section 5 we use results obtained in section 3 to perform computations of the globularily generated piece of double categories presented in section 2 thus providing non-trivial examples of globularily generated double categories. Precisely, we compute the globularily generated piece of double category of oriented cobordisms, double category of algebras, and of double category of semisimple von Neumann algebras.

\

\noindent \textbf{Acknowledgements:} The author would like to thank professor Yasuyuki Kawahigashi, for the trust, support, and encouragement given to the author during his stay at The University of Tokyo. The completion of this paper would have not been possible without him.

\section{Preliminaries}

\noindent In this first section we recall basic concepts and set notational conventions regarding the theory of categorical structures of second order. Further, we introduce the concepts of decorated bicategory, decorated pseudofunctor, and decorated horizontalization. We use this to formally present the problem of existence of internalizations of decorated bicategories, which will serve as motivation for the rest of the work presented in this paper. We present examples relevant to our discussion. 

\

\noindent \textit{Bicategories} 

\

\noindent We refer the reader to [8] and [14] for concepts related to the theory of bicategories. Given a bicategory $B$, we will write $B_0,B_1$, and $B_2$ for the collections of 0-,1-, and 2-cells of $B$ respectively. We will write $i$, $\circ$, and $\ast$ for the horizontal identity functions, and the vertical and horizontal compositions in $B$ respectively. We will not make explicit use of left and right identity transformations or associator of bicategories and will thus omit notational conventions for these concepts. Nevertheless, bicategories will not be assumed to be strict unless explicitly stated.

Given a pseudofunctor $F$ we will write $F_0,F_1$, and $F_2$ for the 0-,1-, and 2-cell components of $F$ respectively. Pseudofunctors will not be assumed to be strict, unless explicitly stated. We will write \textbf{bCat} for the underlying category of the 3-category of bicategories, pseudofunctors, lax or oplax natural transformations, and deformations between transformations.

\

\noindent \textit{Double categories}

\

\noindent We refer the reader to [20] and [22] for concepts related to the theory of double categories. Given a double category $C$ we will write $C_0$ and $C_1$ for the category of objects and the category of morphisms of $C$ respectively. We will denote by $s,t,i$, and $\ast$ the source, target, identity, and horizontal composition functors of a double category $C$. We will again omit notational conventions for identity transformations and associator of double categories. Double categories will not be assumed to be strict unless otherwise stated.

Given a double category $C$, a 2-morphism in $C$ is said to be \textbf{globular} if its source and target are vertical identity endomorphisms. The definition of double category that we use requires for the components of the left and right identity transformations and of the associator of a double category to be globular.

\noindent Given a double functor $F$ we will write $F_0$ and $F_1$ for the object functor and the morphism functor of $F$ respectively. Double functors will not be assumed to be strict unless otherwise stated. Given a double natural transformation $\eta$ we will write $\eta_0$ and $\eta_1$ for the object and morphism components of $\eta$ respectively. We will denote the 2-category of double categories, double functors, and double natural transformations by \textbf{dCat}.

\

\noindent \textit{Trivial double categories}

\

\noindent Given a bicategory $B$, we associate to $B$ a double category $\overline{B}$. The category of objects $\overline{B}_0$ of $\overline{B}$ will be the discrete category $dB_0$ generated by the collection of 0-cells $B_0$ of $B$. Category of morphisms $\overline{B}_1$ of $\overline{B}$ will be the pair formed by the collections of 1- and 2-cells $B_1$ and $B_2$ of $B$, with vertical composition of 2-cells as composition operation. The obvious functors serve as source, target, and identity functors for $\overline{B}$. The horizontal composition bifunctor of $\overline{B}$ will be the bifunctor generated by the horizontal composition operation in $B$. The left and right  identity transformations and the associator of $\overline{B}$ are given by the corresponding constraints in $B$. We call $\overline{B}$ the \textbf{trivial double category} associated to bicategory $B$.

Given a pseudofunctor $F$ between bicategories $B$ and $B'$, we will denote by $\overline{F}$ the pair formed by the functor induced by the 0-cell function $F_0$ of $F$, from the discrete category $dB_0$ generated by $B_0$ to the discrete category $dB_0'$ generated by $B_0'$, and the pair $F_1,F_2$ formed by the 1- and 2-cell functions of $F$. Thus defined the pair $\overline{F}$ is a double functor from the trivial double category $\overline{B}$ associated to $B$ to the trivial double category $\overline{B}'$ associated to $B'$. We call $\overline{F}$ the \textbf{trivial double functor} associated to the pseudofunctor $F$. The pair of functions formed by the function associating to every bicategory $B$ its trivial double category $\overline{B}$ and the function associating to every pseudofunctor $F$ its trivial double functor $\overline{F}$ is a functor from category \textbf{bCat} to category \textbf{dCat}. We call this functor the \textbf{trivial double category functor}. 

\

\noindent \textit{Horizontalization}

\

\noindent Given a double category $C$, we now associate to $C$ a bicategory $HC$ whose collections of 0-,1-, and 2-cells are the collection of objects of $C$, the collection of horizontal morphisms of $C$, and the collection of globular 2-morphisms of $C$ respectively. Vertical and horizontal composition of 2-cells in $HC$ is the vertical and horizontal composition of 2-cells in $C$. The left and right identity transformations and associator of $HC$ are induced by the left and right identity transformations and associator of $C$ respectively. We call $HC$ the \textbf{horizontal bicategory}, or simply the horizontalization of $C$ [20].

Given a double functor $F$ from a double category $C$ to a double category $C'$, we denote by $HF$ the triple formed by the object function of the object functor $F_0$ of $F$, the object function of the morphism functor $F_1$ of $F$, and the restriction of the morphism function of the morphism functor $F_1$ of $F$, to the collection of globular 2-morphisms of $C$. Thus defined $HF$ is a pseudofunctor from the horizontal bicategory $HC$ associated to $C$ to the horizontal bicategory $HC'$ associated to $C'$. We call $HF$ the \textbf{horizontal pseudofunctor}, or the horizontalization, of double functor $F$. The pair $H$ formed by the function associating horizontal bicategory $HC$ to every double category $C$ and the function associating horizontal pseudofunctor $HF$ to every double functor $F$, is a functor from category \textbf{dCat} to \textbf{bCat}. We call this functor the \textbf{horizontalization functor} and we denote it by $H$.
 
\

\noindent \textit{Adjoint relation}

\

\noindent Given a bicategory $B$ it is trivial to see that the equation $H\overline{B}=B$ holds. This relations extends to a functorial setting, that is, the horizontalization functor is left inverse to the trivial double category functor. The horizontalization functor is not right inverse to the trivial double category functor. A weaker relation holds nevertheless. The horizontalization functor is left adjoint to the trivial double category functor. The trivial double category functor is injective and faithful, and thus makes category \textbf{bCat} a subcategory of \textbf{dCat}. The horizontalization functor is thus a reflector for this inclusion. 

\

\noindent \textit{Decorated bicategories}

\

\noindent Given a bicategory $B$ we say that a category $C$ is a decoration of $B$ if the collection of objects of $C$ is equal to the collection of 0-cells of $B$. We will usually indicate that a category $C$ is a decoration of a bicategory $B$ by writing $C$ as $B^*$. We say that a pair $B^*,B$, formed by a category $B^*$ and a bicategory $B$, is a \textbf{decorated bicategory} if category $B^*$ is a decoration of bicategory $B$. In that case we will call $B$ the underlying bicategory of the decorated bicategory $B^*,B$.

Given decorated bicategories $B^*,B$ and $B'^*,B'$ we say that a pair $F,G$ is a \textbf{decorated pseudofunctor} from $B^*,B$ to $B'^*,B'$ if $F$ is a functor from the decoration $B^*$ of $B^*,B$ to the decoration $B'^*$ of $B'^*,B'$, $G$ is a pseudofunctor from the underlying bicategory $B$ of $B^*,B$ to the underlying bicategory $B'$ of $B'^*,B'$, and the object function of $F$ and the 0-cell function of $G$ coincide. Composition of decorated pseudofunctors is performed component-wise. This composition operation is associative and unital and thus provides the pair formed by collection of decorated bicategories and collection of decorated pseudofunctors with the structure of a category. We denote this category by \textbf{bCat}$^*$.

\

\noindent \textit{Decorated horizontalization}

\

\noindent We extend the construction of the horizontalization functor $H$ to a functor from category \textbf{dCat} to category \textbf{bCat}$^*$. Given a double category $C$, we denote by $H^*C$ the pair formed by the object category $C_0$ of $C$ and by the horizontal bicategory $HC$ of $C$. Thus defined, $H^*C$ is a decorated bicategory. Given a double functor $F$ between double categories $C$ and $C'$ we denote by $H^*F$ the pair formed by the object functor $F_0$ of $F$ and by the horizontalization $HF$ of $F$. Thus defined, pair $H^*F$ associated to the double functor $F$ is a decorated pseudofunctor from the decorated horizontalization $H^*C$ of $C$ to the decorated horizontalization $H^*C'$ of $C'$. The function associating, to every double category $C$ the decorated horizontalization $H^*C$ of $C$, together with the function associating, to every double functor $F$ the decorated horizontalization $H^*F$ of $F$, is a functor from category \textbf{dCat} to category \textbf{bCat}$^*$. We denote this functor by $H^*$ and we call it the \textbf{decorated horizontalization functor}.

\

\noindent \textit{Internalization}

\

\noindent Given a decorated bicategory $B^*,B$ we say that a double category $C$ is an \textbf{internalization} of $B^*,B$ if $B^*,B$ is equal to decorated horizontalization $H^*C$ of $C$. Moreover, we say that a functor $\Psi$ from \textbf{bCat}$^*$ to \textbf{dCat} is an \textbf{internalization functor}, if $\Psi$ is a section of decorated horizontalization functor $H^*$. We think of internalization functors as coherent ways of associating internalizations to decorated bicategories
 
\

\noindent \textit{Examples}

\

\noindent \textbf{Trivially decorated bicategories:} Let $B$ be a bicategory. In that case pair $dB_0,B$ formed by the discrete category generated by the collection of 0-cells $B_0$ of $B$, and $B$, is a decorated bicategory. We call the pair $dB_0,B$ the discretely decorated bicategory associated to bicategory $B$. The decorated horizontalization $H^*\overline{B}$ of the trivial double category $\overline{B}$ associated to bicategory $B$ is equal to the discretely decorated bicategory $dB_0,B$ associated to $B$, that is, the trivial double category $\overline{B}$ associated to bicategory $B$ is an internalization of the discretely decorated decorated bicategory $dB_0,B$ associated to $B$.

\

\noindent \textbf{Cobordisms:} Let $n$ be a positive integer. Let \textbf{Cob}$(n)_0$ denote the category whose collection of objects is the collection of closed $n$-dimensional smooth manifolds and whose collection of morphisms is the collection of diffeomorphisms between manifolds. Given closed $n$-manifolds $X,Y,Z$ and $W$, a cobordism $M$ from $X$ to $Y$, and a cobordism $N$ from $Z$ to $W$, we will say that a triple $(f,\Phi,g)$, where $f$ is a diffeomorphisms from $X$ to $Z$, $g$ is a diffeomorphism from $Y$ to $W$, and where $\Phi$ is a diffeomorphism from $M$ to $N$; is an equivariant diffeomorphism from $M$ to $N$ if the restriction of $\Phi$ to $X$ equals $f$ and the restriction of $\Phi$ to $Y$ equals $g$. Composition of equivariant morphisms between cobordisms is performed entry-wise. Let \textbf{Cob}$(n)_1$ denote the category whose collection of objects is the collection of cobordisms between closed $n$-dimensional manifolds and whose collection of morphisms is the collection of equivariant diffeomorphisms between cobordisms. Given an $n$-dimensional manifold $X$ we write $i_X$ for the cobordism $X\times [0,1]$. Given a diffeomorphism $f:X\to Y$ between closed $n$-dimensional manifolds $X$ and $Y$, we denote by $i_f$ the equivariant diffeomorphism $(f,f\times[0,1],f)$ from the cobordism $i_X$ to the cobordism $i_Y$. These two functions define a functor from \textbf{Cob}$(n)_0$ to \textbf{Cob}$(n)_1$. Denote this functor by $i$. Given compatible cobordisms $M$ and $N$ with respect to a manifold $X$ we write $N\ast M$ for the joint union $M\cup_X N$ of $M$ and $N$, with respect to $X$. Finally, given equivariant diffeomorphisms $(f,\Phi,g)$ and $(g,\Psi,h)$, compatible with respect to diffeomorphism $g$ denote by $(g,\Psi,h)\ast(f,\Phi,g)$ the joint union $(f,\Phi\cup_g\Psi,h)$ of $(f,\Phi,g)$ and $(g,\Psi,h)$ with respect to $g$. The two operations defined above form a bifunctor which we denote by $\ast$. With this structure the pair \textbf{Cob}$(n)$ is a double category, where the identity transformations and associator come from the obvious diffeomorphisms from cobordism theory. The decorated horizontalization $H^*$\textbf{Cob}$(n)$ of \textbf{Cob}$(n)$ is equal to bicategory $\underline{\mbox{\textbf{Cob}}(n)}$ of $n$-dimensional cobordisms and diffeomorphisms presented in [17] decorated by the grupoid of diffeomorphisms of closed oreiented $n$-dimensional manifolds, that is the double category \textbf{Cob}$(n)$ is an internalization of bicategory $\underline{\mbox{\textbf{Cob}}(n)}$ decorated by the grupoid of diffeomorphisms of closed oriented $n$-dimensional manifolds.

\

\noindent \textbf{Algebras:} Let \textbf{Alg}$_0$ denote the category whose collection of objects is the collection of complex algebras and whose collection of morphisms is the collection of unital algebra morphisms. Given algebras $A,B,C$ and $D$, a left-right $A$-$B$ bimodule $M$ and a left-right $C$-$D$-bimodule $N$, we say that a triple $(f,\Phi,g)$, where $f$ is a unital algebra morphism from $A$ to $C$, $g$ is a unital algebra morphism from $B$ to $D$ and $\Phi$ is a linear transformation from $M$ to $N$; is an equivariant bimodule morphism from $M$ to $N$, if for every $a\in A$, $b\in B$, and $x\in M$ equation $\Phi(axb)=f(a)\Phi(x)g(b)$ holds. Composition of equivariant bimodule morphisms is performed entry-wise. Let \textbf{Alg}$_1$ denote the category whose collection of objects is the collection of bimodules over complex algebras and whose collection of morphisms is the collection of equivariant bimodule morphisms. Denote by \textbf{Alg} the pair formed by categories \textbf{Alg}$_0$ and \textbf{Alg}$_1$. Denote by $i$ the functor from \textbf{Alg}$_0$ to \textbf{Alg}$_1$, associating $A$ considered as a left-right $A$-bimodule to every algebra $A$, and associating the equivariant bimodule morphism $(f,f,f)$ to every unital algebra morphism $f$. Given algebras $A,B$ and $C$, a left-right $A$-$B$ bimodule $M$ and a left-right $B$-$C$ bimodule $N$, write $N\ast M$ for the tensor product $M\otimes_B N$ relative to $B$. Given equivariant bimodule morphisms $(f,\Phi,g)$ and $(g,\Psi,h)$ we write $(g,\Psi,h)\ast(f,\Phi,g)$ for the triple $(f,\Phi\otimes_g\Psi,h)$. Denote by $\ast$ the bifunctor, from \textbf{Alg}$_1\times_{\mbox{\textbf{Alg}}_0}$\textbf{Alg}$_1$ to \textbf{Alg}$_1$ defined by operations $\ast$ defined above. This structure provides the pair \textbf{Alg} of categories \textbf{Alg}$_0$ and \textbf{Alg}$_1$ with the structure of a double category. Left and right identity transformations and associator in \textbf{Alg} are defined by those associated to the relative tensor product bifunctor. The decorated horizontalization $H^*$\textbf{Alg} of \textbf{Alg} is equal to the bicategory $\underline{\mbox{\textbf{Alg}}}$ of unital algebras, algebra bimodules, and bimodules morphisms presented in [17], decorated by category \textbf{Alg}$_0$ of unital algebra morphisms. That is, the double category \textbf{Alg} is an internalization of bicategory $\underline{\mbox{\textbf{Alg}}}$, decorated by the category \textbf{Alg}$_0$ of algebras, and unital algebra morphisms.  

\

\noindent \textbf{von Neumann algebras:} Let $A$ be a von Neumann algebra. We say that $A$ is semisimple if the center $A\cap A'$ of $A$ is finite dimensional. Let $A$ and $B$ be semisimple von Neumann algebras. Let $f:A\to B$ be a von Neumann algebra morphism from $A$ to $B$. Let $A=\bigoplus_{i=1}^nA_i$ and $B=\bigoplus_{j=1}^mB_j$ be direct decompositions of $A$ and $B$ such that $A_j$ is a factor for every $1\leq i\leq n$ and such that $B_j$ is a factor for every $1\leq j\leq m$. Write $f$ in matrix form as $(f_{i,j})$ where $f_{i,j}$ is a von Neumann algebra morphism from factor $A_i$ to factor $B_j$ for every $i,j$. In that case we say that the morphism $f$ is finite if morphism $f_{i,j}$ has finite index [13],[15],[16] for every $1\leq i\leq n$ and $1\leq j\leq m$. The composition of finite morphisms between semisimple von Neumann algebras is again a finite morphism. We denote by $[W^*]^{f}_0$ the category whose objects are semisimple von Neumann algebras and whose morphisms are finite morphisms. Let $A,B,C$, and $D$ be semisimple von Neumann algebras. Let $M$ be a normal $A$-$B$-bimodule and let $N$ be a normal $C$-$D$-bimodule. We say that a triple $(f,\Phi,g)$ is an equivariant finite intertwiner from $M$ to $N$ if $f$ is a finite morphism from $A$ to $C$, $\Phi$ is a bounded operator from $M$ to $N$, $g$ is a finite morphism from $B$ to $D$, and if equation $\Phi(axb)=f(a)\Phi(x)g(b)$ holds for every $a\in A$, $x\in M$ and $b\in B$. Composition of equivariant finite morphisms is performed entry-wise. Denote by $[W^*]^f_1$ the category whose objects are bimodules over semisimple von Neumann algebras and whose morphisms are bounded intertwiners. Denote by $[W^*]^f$ the pair formed by $[W^*]^f_0$ and $[W^*]^f_1$.
Let $A$ be a semisimple von Neumann algebra. Denote by $i$ the functorial extension of the Haagerup standard form construction [12] defined in [4]. Thus defined $i$ is a functor from category $[W^*]^f_0$ to category $[W^*]^f_1$. Denote now by $\ast$ the functorial extension of the Connes fusion product operation [9],[18] defined in [4]. Thus defined $\ast$ is a bifunctor from category $[W^*]_1^f\times_{[W^*]_0^f}[W^*]_1^f$ to category $[W^*]_1^f$. With this structure pair $[W^*]^f$ forms a double category. Denote by $\underline{[W^*]}^f$ the sub-bicategory of bicategory $\underline{[W^*]}$ of von Neumann algebras, normal bimodules, and bimodule intertwiners defined in [17], generated by semisimple von Neumann algebras. In that case the decorated horizontalization $H^*[W^*]^f$ of double category $[W^*]^f$ is equal to bicategory $\underline{[W^*]}^f$, decorated by category $[W^*]^f_0$ of semisimple von Neumann algebras and finite morphisms, that is, double category $[W^*]^f$ is an internalization of bicategory $\underline{[W^*]}^f$ decorated by category $[W^*]^f_0$.

\

\noindent\textit{The problem of existence of internalizations}

\

\noindent It is not known if bicategory $\underline{[W^*]}$ of general von Neumann algebras, normal bimodules, and bimodule intertwiners appearing in [17], decorated by the category of general von Neumann algebras and general von Neumann algebra morphisms, admits an internalization. The existence of extensions of functors $i$ and $\ast$ appearing in the definition of $[W^*]^f$, to the category of general von Neumann algebras and general von Neumann algebra morphisms and to the category of normal bimodules and general equivariant bimodule intertwiners would imply the existence of such an internalization. The existence of such extensions of $i$ and $\ast$ is conjectured in [4].

We are interested in the problem of finding internalizations for general decorated bicategories. We refer to this problem as the \textbf{problem of existence of internalizations for decorated bicategories}. The problem of existence of internalizations for decorated bicategories is the main motivation for the work at present. We will moreover be interested in the problem of finding internalization functors. We refer o this problem as \textbf{the problem of existence of internalization functors} and we regard it as a categorical version of the problem of existence of internalizations for decorated bicategories.

\section{Globularily generated double categories}

\noindent In this section we introduce globularily generated double categories. We prove that the condition of a double category being globularily generated is minimal with respect to having a given decorated bicategory as decorated horizontalization and that the problem of existence of internalizations for decorated bicategories is equivalent to the problem of existence of globularily generated internalizations. We do this by associating, to every double category, a globularily generated double category, its globularily generated piece. We furnish the globularily generated piece construction with the structure of a 2-functor. We extend results on the globularily generated piece construction to analogous results regarding internalization functors. Finally, we prove that globularily generated piece 2-functor is a 2-reflector. We interpret this by regarding the condition of a double category being globularily generated as a categorical analog of the condition of a double category being trivial.

\

\noindent Given a double category $C$, a \textbf{sub-double category} of $C$ is a double category satisfying the standard notion of sub-structure applied to $C$. More precisely, we will say that a double category $D$ is a sub-double category of a double category $C$ if the object and morphism categories $D_0$ and $D_1$ of $D$ are subcategories of the object and morphism categories $C_0$ and $C_1$ of $C$ respectively, if the structure functors $s,t,i$ and $\ast$ of $C$ restrict to those of $D$, and if the the left and right identity transformations and associator for $C$ restrict to the left and right identity transformations and associator for $D$ respectively. 

Given a double category $C$, we will say that a sub-double category $D$ of $C$ is \textbf{complete}, if the collections of objects, vertical morphisms, and horizontal morphisms of $D$ are equal to the collections of objects, vertical morphisms, and horizontal morphisms of $C$ respectively. 

Given a double category $C$, it is easily seen that the collection of complete sub-double categories of $C$ is closed under the operation of taking arbitrary intersections. Given a collection of 2-morphisms $X$ of a double category $C$, we will call the intersection of all complete sub-double categories $D$, of $C$, such that $X$ is contained in the collection of 2-morphisms of $D$, the \textbf{complete sub-double category} of $C$ generated by $X$. We will say that a collection of 2-morphisms $X$ of double a category $C$ generates $C$ if $C$ is equal to the complete sub-double category of $C$ generated by $X$. The following is the main definition of this section.

\begin{definition}
Let $C$ be a double category. We say that $C$ is globularily generated if $C$ is generated by its collection of globular 2-morphisms.
\end{definition}

\noindent Trivial double categories are examples of globularily generated double categories. Non-trivial examples will be provided in section 6.

\

\noindent Given a double category $C$, we write $\gamma C$ for the complete sub-double category of $C$ generated by the collection of globular 2-morphisms of $C$. Thus defined, $\gamma C$ is a globularily generated complete sub-double category of $C$. We call double category $\gamma C$ associated to $C$ the \textbf{globularily generated piece} of $C$. Thus defined, the globularily generated piece $\gamma C$ of double category $C$ is equal to both the maximal globularily generated sub-double category of $C$ and to the minimal complete sub-double category of $C$ containing the collection of globular 2-morphisms of $C$.

\begin{prop}
Let $C$ be a double category. The decorated horizontalization $H^*C$ of $C$ is equal to the decorated horizontalization $H^*\gamma C$ of the globularily generated piece $\gamma C$ of $C$. Moreover, the globularily generated piece $\gamma C$ of $C$ is minimal with respect to this property.
\end{prop}

\begin{proof}
Let $C$ be a double category. The fact that the decorated horizontalization $H^*C$ of $C$ and the decorated horizontalization $H^*\gamma C$ of the globularily generated piece $\gamma C$ of $C$ are equal follows from the fact that the globularily generated piece $\gamma C$ of $C$ is complete in $C$ and from the easy observation that the collection of globular 2-morphisms of both $C$ and $\gamma C$ are equal. Now, suppose $D$ is a sub-double category of $C$ such that the decorated horizontalization $H^*C$ of $C$ is equal to the decorated horizontalization $H^*C$ of $C$. We wish to prove that the globularily generated piece $\gamma C$ of $C$ is a sub-double category of $D$. First observe that from the fact that $H^*C$ and $H^*D$ are equal it follows that $D$ is complete in $C$. It follows that both the category of objects and the collection of horizontal morphisms of both $D$ and $\gamma C$ are equal. Again from the equation $H^*C=H^*D$ it follows that the collection of globular 2-morphisms of $C$ is contained in the collection of 2-morphisms of $D$. We conclude that the globularily generated piece $\gamma C$ is a sub-double category of $D$. This concludes the proof. 
\end{proof}

\noindent A consequence of proposition 3.2 is that the problem of existence of internalizations for a given decorated bicategory $B^*,B$ is equivalent to the problem of existence of globularily generated internalizations for $B^*,B$. The following corollary says that the condition of a double category being globularily generated is precisely the condition of a double category being a minimal internalization.

\begin{cor}
Let $C$ be a double category. $C$ is globularily generated if and only if $C$ is minimal amongst internalizations of decorated horizontalization $H^*C$ of $C$.
\end{cor}

\noindent We interpret corollary 3.3 by saying that the condition of a double category being globularily generated is precisely the condition of a double category being minimal with respect to the decorated horizontalization functor. We regard this and proposition 3.2 as the main motivation for the study of globularily generated double categories.

\

\noindent We extend the construction of the globularily generated piece of a double category, to double functors as follows: Given a double functor $F:C\to D$ from a double category $C$ to a double category $D$, the image, under the morphism functor $F_1$ of $F$, of every globular 2-morphism in $C$, is a globular 2-morphism in $D$. It follows that the restriction of double functor $F$, to the globularily generated piece $\gamma C$ of $C$ defines a double functor from the globularily generated piece $\gamma C$ of $C$ to the globularily generated piece $\gamma D$ of $D$. We write $\gamma F$ for this double functor. We call $\gamma F$ the \textbf{globularily generated piece double functor of} $F$. We regard the following proposition as a functorial version of proposition 3.2.

\begin{prop}
Let $C,D$ be double categories. Let $F:C\to D$ be a double functor. The decorated horizontalization $H^*F$ of $F$ is equal to the decorated horizontalization $H^*\gamma F$ of the globularily generated piece $\gamma F$ of $F$. Moreover, the globularily generated piece $\gamma F$ of $F$ is minimal with respect to this property.
\end{prop}

\begin{proof}
Let $C,D$ be double categories. Let $F:C\to D$ be a double functor. From the fact that the object functor and the object function of the morphism functor of the globularily generated piece $\gamma F$ of $F$ are equal to the object functor and the object function of the morphism functor of $F$ respectively, and from the fact that the globularily generated piece $\gamma F$ of $F$ is equal to $F$ on globular 2-morphisms of $C$ it follows that the decorated horizontalization $H^*F$ of $F$ is equal to the decorated horizontalization $H^*\gamma F$ of the globularily generated piece $\gamma F$ of $F$. Now let $C',D'$ be sub-double categories of $C$ and $D$ respectively. Let $F':C'\to D'$ be a double sub-functor of $F$ such that $H^*F$ and $H^*F'$ are equal. We wish to prove, in that case, that the globularily generated piece $\gamma F$ of $F$ is a sub-double functor of $F'$. The globularily generated piece $\gamma C$ is a sub-double category of $C'$ and the globularily generated piece $\gamma D$ of $D$ is a sub-double category of $D'$ by proposition 3.2. By equation $H^*F'=H^*F$ it follows that the object functor and the object function of morphism functor of $F'$ are equal to the object functor and the object function of morphism functor of $\gamma F$ respectively. From this and from the fact that the restriction of $F'$ to the collection of globular 2-morphisms of $C$ is equal to the restriction of $F$ to the collection of globular 2-morphisms of $C$ it follows that $\gamma F$ is a subfunctor of $F'$. This concludes the proof.
\end{proof}

\noindent We write \textbf{gCat} for the subcategory of the underlying category of 2-category \textbf{dCat} of double categories, double functors, and double natural transformations, generated by globularily generated double categories. The function associating the globularily generated piece $\gamma C$ to every double category $C$ and the globularily generated piece double functor $\gamma F$ to every double functor $F$, is a functor from the underlying category of 2-category \textbf{dCat} to \textbf{gCat}. We denote this functor by $\gamma$. We call functor $\gamma$ the \textbf{globularily generated piece functor}. We consider the globularily generated piece functor $\gamma$ as a categorification of the globularily generated piece construction. 

\

\noindent We will say that an internalization functor $\Psi$, as defined in the previous section, is a \textbf{globularily generated internalization functor}, if the image category of $\Psi$ lies in category \textbf{gCat}. A consequence of proposition 3.4 is that the composition $\gamma \Phi$ of an internalization functor $\Phi$ and the globularily generated piece functor $\gamma$, is again an internalization functor, and thus the problem of existence of internalization functors is equivalent to the problem of existence of globularily generated internalization functors. The following corollary is a functorial analog of corollary 3.3.

\begin{cor}
Let $\Phi$ be an internalization functor. In that case $\Phi$ is globularily generated if and only if $\Phi$ does not properly contain internalization sub-functors.
\end{cor}

\noindent We extend the globularily generated piece functor to a 2-functor between appropiate 2-categories as follows: Given a double category $C$, we call 2-morphisms in $C$ lying in globularily generated piece $\gamma C$ of $C$ the \textbf{globularily generated 2-morphisms} of $C$. Given double functors $F,G:C\to D$ from a double category $C$ to a double category $D$ we say that a double natural transformation $\eta:F\to G$, from $F$ to $G$, is a \textbf{globularily generated double natural transformation}, if every component of the morphism part $\eta_1$ of $\eta$ is a globularily generated 2-morphism.

Identity double natural transformations are examples of globularily generated double natural transformations. Further, the collection of globularily generated double natural transformations is closed under the operations of taking vertical and horizontal compositions. It follows that the triple formed by collection of double categories, collection of double functors, and collection of globularily generated double natural transformations forms a sub 2-category of 2-category \textbf{dCat}. We denote this 2-category by \textbf{dCat}$^{\mbox{g}}$.

We keep denoting by \textbf{gCat} the full sub 2-category of \textbf{dCat}$^{\mbox{g}}$ generated by globularily generated double categories. 2-category \textbf{gCat} now has collection of globularily generated double categories, collection of double functors between globularily generated double categories, and collection of globularily generated double natural transformations as collections of 0-, 1-, and 2-cells respectively. The globularily generated piece functor $\gamma$ admits a unique extension to a 2-functor from \textbf{dCat}$^g$ to \textbf{gCat} such that this extension acts as the identity function on globularily generated double natural transformations. We keep denoting this extension by $\gamma$.

\

\noindent Given 2-categories $B$ and $B'$, such that $B$ is a full sub 2-category of $B'$, we will say that $B$ is a \textbf{strictly 2-reflective sub 2-category} of $B'$ if the inclusion 2-functor of $B$ in $B'$ admits an left adjoint 2-functor with counit and unit being strict 2-natural transformations. In that case we will say that any 2-functor, left adjoint to the inclusion 2-functor of $B$ in $B'$, is a \textbf{strict 2-reflector} of $B'$ on $B$. The next proposition says that 2-category \textbf{gCat}, as a sub 2-category of 2-category \textbf{dCat}$^g$ is strictly 2-reflective, with the globularily generated piece 2-functor $\gamma$ as 2-reflector.

\begin{prop}
2-category \textbf{gCat} is a strictly 2-reflective sub 2-category of 2-category \textbf{dCat}$^g$ with globularily generated piece 2-functor $\gamma$ as 2-reflector.
\end{prop}

\begin{proof} Let $i$ denote the inclusion 2-functor of 2-category \textbf{gCat} in 2-category \textbf{dCat}$^g$. We wish to provide pair $(\gamma,i)$ formed by the globularily generated piece 2-functor $\gamma$ and the inclusion 2-functor $i$ with the structure of an adjoint pair. We associate, to pair $(\gamma,i)$ a counit-unit pair $(\epsilon,\eta)$.

Let $C$ be a double category. Write $\epsilon_C$ for the inclusion double functor of the globularily generated piece $\gamma C$ associated to $C$ in $C$. We write $\epsilon$ for the collection of inclusions $\epsilon_C$ where $C$ runs through collection of all double categories. We prove that thus defined $\epsilon$ is a 2-natural transformation [14] from the composition $i\gamma$ of the globularily generated piece 2-functor $\gamma$ and inclusion $i$ to the identity 2-endofunctor $id_{\mbox{\textbf{dCat}}^g}$ of 2-category \textbf{dCat}$^g$. Let $C$ and $D$ be double categories. Let $F:C\to D$ be a double functor from $C$ to $D$. Since double functor $\gamma$ is defined by restriction on 1-cells, the following square

\begin{center}

\begin{tikzpicture}
\matrix(m)[matrix of math nodes, row sep=3.6em, column sep=6em,text height=1.5ex, text depth=0.25ex]
{i\gamma C&i\gamma D\\
C&D\\};
\path[->,font=\scriptsize,>=angle 90]
(m-1-1) edge node[auto] {$i\gamma F$} (m-1-2)
        edge node[left] {$\epsilon_C$} (m-2-1)
(m-2-1) edge node[below] {$F$} (m-2-2)
(m-1-2) edge node[right] {$\epsilon_D$} (m-2-2);
\end{tikzpicture}

\end{center}

\noindent commutes. Now, let $F,G:C\to D$ be double functors from double category $C$ to double category $D$ and let $\mu:F\to G$ be a globularily generated natural transformation from $F$ to $G$. We wish to prove that equation

\[\epsilon_D\mu=\mu\epsilon_C\]

\noindent holds. The above equation is equivalent to the following pair of equations

\[\epsilon_{D_0}\mu_0=\mu_0\epsilon_{C_0} \ \mbox{and} \ \epsilon_{D_1}\mu_1=\mu_1\epsilon_{C_1}\]

\noindent Observe that since the globularily generated piece 2-functor acts as the identity 2-functor on object categories, object functors, and object natural transformations, the first equation above is trivial. We thus need only to prove that the second equation holds. Let $\alpha$ be a horizontal morphism in $C$. In that case $\mu_\alpha$ is a globularily generated 2-morphism and thus $\epsilon_D\mu_\alpha$ is equal to $\mu_\alpha$. Now, $\epsilon_C\alpha$ is equal to $\alpha$ and thus $\mu\epsilon_C\alpha$ is equal to $\mu_\alpha$. We conclude that both equations above hold and thus $\epsilon$ is a strict 2-natural transformation from the composition $i\gamma$ of the globularily generated piece 2-functor $\gamma$ and inclusion $i$ to the identitiy 2-endofunctor $id_{\mbox{\textbf{dCat}}^g}$ of 2-category \textbf{dCat}$^g$. 

Now, since the globularily generated piece of a globularily generated double category is equal to the original globularily generated category and the globularily generated piece double functor $\gamma$ acts by restriction on double functors and double natural transformations, the composition $\gamma i$ of the inclusion $i$ of 2-category \textbf{gCat} in 2-category \textbf{dCat}$^g$ and the globularily generated piece 2-functor $\gamma$ is equal to the identity 2-endofunctor of 2-category \textbf{gCat}. Denote by $\eta$ the identity double natural transformation of the identity 2-endofunctor $id_{\mbox{\textbf{gCat}}}$ of \textbf{gCat} as a natural transformation from $id_{\mbox{\textbf{gCat}}}$ to the composition $\gamma i$. Thus defined $\eta$ is a strict natural transformation. Finally, observe that from the way $\eta$ was defined, pair of natural transformations $(\epsilon,\eta)$ clearly satisfies the counit-unit triangle equations and it is thus a counit-unit pair for pair $(\gamma,i)$. We conclude that 2-category \textbf{gCat} is a reflective 2-subcategory of 2-category \textbf{dCat}$^g$ and that 2-functor $\gamma$ acts as a reflector.
\end{proof}

\noindent We interpret proposition 3.6 by considering the globularily generated piece 2-functor $\gamma$ as an analog of the horizontalization functor $H$ and thus regarding 2-category \textbf{gCat} as an analog of category \textbf{bCat} in 2-category \textbf{dCat}$^g$.

\

\noindent In the following corollary, for a given double category $C$, we keep denoting, as in the proof of proposition 3.6, the inclusion double functor from the globularily generated piece $\gamma C$ of $C$ to $C$ by $\epsilon_C$

\begin{cor}
Let $C$ be a double category. The globularily generated piece $\gamma C$ of $C$ is characterized up to double isomorphisms by the following universal property: Let $D$ be a globularily generated double category. Let $F:D\to C$ be a double functor from $D$ to $C$. In that case there exists a unique double functor $\tilde{F}:D\to \gamma C$ from $D$ to globularily generated piece $\gamma C$ of $C$ such that the following triangle commutes:

\begin{center}

\begin{tikzpicture}
\matrix(m)[matrix of math nodes, row sep=3.6em, column sep=6em,text height=1.5ex, text depth=0.25ex]
{D &  C\\
\gamma C & \\};
\path[->,font=\scriptsize,>=angle 90]
(m-1-1) edge node[auto] {$F$} (m-1-2)
        edge node[left] {$\tilde{F}$} (m-2-1)
(m-2-1) edge node[below] {$\epsilon_C$} (m-1-2);
\end{tikzpicture}

\end{center}

\end{cor}

\section{Structure}

\noindent In this section we introduce technical tools in the theory of globularily generated double categories. We define the vertical length of a globularily generated 2-morphism and we use this to obtain general information on the structure of globularily generated double categories. We prove in particular that the condition of a double category being globularily generated is not trivial. We begin by recursively associating, to every globularily generated double category, a sequence of subcategories of its category of morphisms as follows:

\

\noindent Let $C$ be a globularily generated double category. Denote by $H^C_1$ the union of collection of globular 2-morphisms of $C$ and collection of horizontal identities of vertical morphisms of $C$. Write $V^C_1$ for the subcategory of the category of morphisms $C_1$ of $C$, generated by $H^C_1$, that is, $V^C_1$ denotes the subcategory of $C_1$ whose morphisms are vertical compositions of globular 2-morphisms and horizontal identities of $C$.

Let $n$ be an integer strictly greater than 1. Suppose that category $V^C_{n-1}$ has been defined. We now define category $V^C_n$. First denote by $H^C_n$ the collection of all possible horizontal compositions of 2-morphisms in category $V^C_{n-1}$. We make, in that case, category $V^C_n$ to be the subcategory of the category of morphisms $C_1$ of $C$, generated by $H^C_n$. That is, category $V^C_n$ is the subcategory of $C_1$ whose collection of morphisms is the collection of vertical compositions of elements of $H^C_n$. 

\

\noindent We have thus associated, to every double category $C$, a sequence of subcategories $\left\{V^C_n\right\}$ of the category of morphisms $C_1$ of $C$. We call, for every $n$, category $V^C_n$ the $n$\textbf{-th vertical category} associated to $C$. We have used, in the above construction, for every $n$, an auxiliary collection of 2-morphisms $H^C_n$ of $C$. Observe that for each $n$, collection $H^C_n$ both contains the horizontal identity of every vertical morphism in $C$ and is closed under the operation of taking horizontal compositions. If double category $C$ is strict, then, for every $n$, collection $H^C_n$ is the collection of morphisms of a category whose collection of objects is the collection of vertical morphisms of $C$. In that case we call category $H^C_n$ the $n$\textbf{-th horizontal category} associated to $C$.

\

\noindent By the way the sequence of vertical categories $\left\{V^C_n\right\}$ associated to a double category $C$ was constructed it is easily seen that for every $n$, inclusions

 \[\mbox{Hom}V^C_n\subseteq H^C_{n+1}\subseteq \mbox{Hom}V^C_{n+1}\]

\noindent hold. This implies that the $n$-th vertical category $V^C_n$ associated to double category $C$ is a subcategory of the $n+1$-th vertical category $V^C_{n+1}$ associated to $C$ for every $n$. Moreover, in the case in which double category $C$ is strict, the $n$-th horizontal category $H^C_n$ associated to $C$ is a subcategory of the $n+1$-th horizontal category $H^C_{n+1}$ associated to $C$ for every $n$. 

\

\noindent The following lemma says that the sequence of vertical categories of a globularily generated double category forms a filtration of its category of morphisms.

\pagebreak

\begin{lem}
Let $C$ be a globularily generated double category. Morphism category $C_1$ of $C$ is equal to the limit $\varinjlim V^C_n$ in the underlying category of \textbf{Cat}, of sequence $\left\{V^C_n\right\}$ of vertical categories associated to $C$.
\end{lem}  

\begin{proof}
Let $C$ be a globularily generated double category. We wish to prove that morphism category $C_1$ of $C$ is equal to the limit $\varinjlim V^C_n$ of the sequence of vertical categories associated to $C$.

By the way the sequence of vertical categories associated to $C$ was constructed, it is easily seen that the union $\bigcup_{n=1}^\infty \mbox{Hom}V^C_n$ of the sequence of its collections of morphisms is closed under the operations of taking vertical and horizontal compositions in $C$, and that it contains the collection of horizontal identities of vertical morphisms of $C$. It follows, from this and from the requirement that the components of the identity transformations and associator of $C$ are globular, that pair formed by the category of objects $C_0$ of $C$ and the limit $\varinjlim V^C_n$ of the sequence of vertical categories associated to $C$ is a sub-double category of $C$. The collection of objects of the $n$-th vertical category $V^C_n$ associated to $C$ is equal to the collection of horizontal morphisms of $C$ for every $n$. Thus the collection of objects of category $\varinjlim V^C_n$ is equal to the collection of horizontal morphisms of $C$. Pair formed by the category of objects $C_0$ of $C$ and the limit $\varinjlim V^C_n$ of the sequence of vertical categories associated to $C$ is thus a complete sub-double cateogory of $C$. Moreover, the collection of morphisms $\bigcup_{n=1}^\infty\mbox{Hom}V^C_n$ of category $\varinjlim V^C_n$ contains the collection of globular 2-morphisms of $C$. We conclude, from this, and from the fact that double category $C$ is globularily generated, that the sub-double category of $C$, formed by $C_0$ and $\varinjlim V^C_n$, is equal to $C$ and thus morphism category $C_1$ of $C$ equals the limit $\varinjlim V^C_n$ of the sequence of vertical categories associated to $C$. This concludes the proof.  
\end{proof}

\noindent Given a strict double category $C$ the pair $\tau C$ formed by collection of vertical morphisms of $C$ and collection of 2-morphisms of $C$ is a category. Composition operation in $\tau C$ is horizontal composition in $C$. We call category $\tau C$ associated to a strict double category $C$ the \textbf{transversal category associated to} $C$. 

\begin{cor}
Let $C$ be a globularily generated double category. If $C$ is a strict double category, then the transversal category $\tau C$ associated to $C$ is equal to the limit $\varinjlim H^C_n$, in the underlying category of \textbf{Cat}, of the sequence of horizontal categories associated to $C$.
\end{cor}

\begin{proof}
Let $C$  be a strict globularily generated double category. We wish to prove, in this case, that the transversal category $\tau C$ associated to $C$ is equal to the limit $\varinjlim H^C_n$ of the sequence of horizontal categories associated to $C$.

By the way the sequence of horizontal categories associated to globularily generated double category $C$ was constructed, it is easily seen that the collection of objects of the $n$-th horizontal category $H^C_n$ associated to $C$ is equal to the collection of vertical morphisms of $C$ for every $n$. It follows, from this, that the collection of objects of the limit $\varinjlim H_n^C$ of the sequence of horizontal categories associated to $C$ is equal to the collection of vertical morphisms of $C$ and is thus equal to the collection of objects of the transversal category $\tau C$ of $C$. The collection of morphisms of the limit $\varinjlim H^C_n$ is equal to the union $\bigcup_{n=1}^\infty\mbox{Hom}H^C_n$ of the collections of morphisms of horizontal categories associated to $C$. This union is equal to the union $\bigcup_{n=1}^\infty\mbox{Hom}V_n^C$ of the collections of morphisms of vertical categories associated to $C$, which by lemma 4.1 is equal to the collection of 2-morphisms of $C$. This concludes the proof.
\end{proof}

\begin{definition}
Let $C$ be a globularily generated double category. Let $\Phi$ be a 2-morphism in $C$. We call the minimal integer $n$ such that $\Phi$ is a morphism of the $n$-th vertical category $V^C_n$ associated to $C$ the vertical length of $\Phi$.
\end{definition}

\noindent We now apply the concept of vertical length to the proof of results concerning the structure of globularily generated double categories. We first establish notational conventions.

\

\noindent Assuming a double category $C$ is strict, the horizontal composition $\Phi_k\ast...\ast\Phi_1$ of any composable sequence $\Phi_1,...,\Phi_k$ of 2-morphisms in $C$, is unambiguously defined. This is not the case in general. If a double category $C$ is not assumed to be strict, then the horizontal compositions of a composable sequence $\Phi_1,...,\Phi_k$ of 2-morphisms in $C$, following different parentheses patterns, might yield different 2-morphisms. If a 2-morphism $\Phi$ in a double category $C$ can be obtained as the horizontal composition, following a certain parentheses pattern, of composable sequence of 2-morphisms $\Phi_1,...,\Phi_k$, we will write $\Phi\equiv\Phi_k\ast...\ast\Phi_1$. 

\

\noindent Given 2-morphisms $\Phi$ and $\Psi$ in a double category $C$, we say that $\Phi$ and $\Psi$ are \textbf{globularily equivalent} if there exist globular 2-isomorphisms $\Theta_1,\Theta_2$ in $C$ such that the equation $\Phi=\Theta_1\Psi\Theta_2^{-1}$ holds.

From the fact that associators in double cateogries satisfy the pentagon axiom and from the fact that the collection of globular 2-morphisms of any double category is closed under the operations of taking vertical and horizontal composition, it follows that if two 2-morphisms $\Phi$ and $\Psi$ satsify equation $\Phi,\Psi\equiv\Phi_k\ast...\ast\Phi_1$ for a composable sequence of 2-morphisms $\Phi_1,...,\Phi_k$ in $C$, then $\Phi$ and $\Psi$ are globularily equivalent.

\

\noindent Finally, we will say that a 2-morphism $\Phi$ in a double category $C$ is a \textbf{horizontal endomorphism} if the source and target $s\Phi,t\Phi$ of $\Phi$, are equal. Horizontal identities are examples of horizontal endomorphisms. The next proposition says that a 2-morphism in a globular double category is either globular or a horizontal endomorphisms.

\begin{prop}
Let $C$ be a globularily generated double category. Let $\Phi$ be a 2-morphism in $C$. If $\Phi$ is not globular then $\Phi$ is a horizontal endomorphism.
\end{prop}

\begin{proof}
Let $C$ be a globularily generated double category. Let $\Phi$ be a non-globular 2-morphism in $C$. We wish to prove that $\Phi$ is a horizontal endomorphism.

We proceed by induction on the vertical length of $\Phi$. Suppose first that $\Phi$ is an element of $H^C_1$. In that case, by the assumption that $\Phi$ is non-globular, $\Phi$ must be the horizontal identity of a vertical morphism in $C$ and thus must be a horizontal endomorphism. Suppose now that $\Phi$ is a general element of the first vertical category $V^C_1$ associated to $C$. Write $\Phi$ as a vertical composition $\Phi=\Phi_k\circ...\circ\Phi_1$ where $\Phi_i$ is an element of $H^C_1$ for every $i$. Moreover, assume that the length $k$ of this decomposition is minimal. We prove by induction on $k$ that $\Phi$ must be a horizontal endomorphism. Suppose first that $k=1$. In that case $\Phi$ is an element of $H^C_1$ and is thus a horizontal identity. Suppose now that $k$ is strictly greater than 1 and that the result is true for every 2-morphism in the first vertical category $V^C_1$ associated to $C$ that can be written as a vertical composition of strictly less than $k$ 2-morphisms in $H^C_1$. Write $\Psi$ for composition $\Phi_k\circ...\circ\Phi_2$. In this case equation $\Phi=\Psi\circ \Phi_1$ holds. Now, since the collection of globular 2-morphisms of $C$ is closed under the operation of taking vertical composition, one of $\Psi$ and $\Phi_1$ is not globular. If both $\Psi$ and $\Phi_1$ are not globular, then by induction hypothesis both $\Psi$ and $\Psi_1$ are horizontal endomorphisms and thus their vertical composition $\Phi$ is a horizontal endomorphism. Suppose now that $\Psi$ is globular. In that case $\Phi_1$ is a horizontal endomorphism. Now, from the fact that source and target of $\Psi$ are in this case vertical identities and from the fact that source and target are functorial, equations $s\Phi=s\Phi_1$ and $t\Phi=t\Psi_1$ follow and thus $\Phi$ is a horizontal endomorphism. The case in which $\Phi_1$ is globular is handled analogously. This concludes the base of the induction.

Let $n$ be a positive integer strictly greater than 1. Assume now that every non-globular 2-morphism in $C$ of vertical length strictly less than $n$ is a horizontal endomorphism. Suppose first that $\Phi$ is an element of $H^C_n$. Let $\Phi_k\ast...\ast\Phi_1$ represent a horizontal composition in $C$ such that $\Phi_i$ is a morphism of $V^C_{n-1}$ for each $i$ and such that $\Phi\equiv\Phi_k\ast...\ast\Phi_1$. Suppose that the length $k$ of this decomposition is minimal. We proceed by induction over $k$. If $k=1$ then $\Phi$ is an element of $V^C_{n-1}$ and is thus a horizontal endomorphism by induction hypothesis. Suppose now that $k$ is strictly greater than 1 and that the result is true for every non-globular 2-morphism in $H^C_n$ that can be written as a horizontal composition of strictly less than $k$ 2-morphisms in $V^C_{n-1}$. Choose $\Psi$ such that $\Psi\equiv\Phi_k\ast...\ast\Phi_2$. In this case $\Phi$ and $\Psi\ast\Phi_1$ are globularily equivalent and thus have the same source and target. Now, if both $\Psi$ and $\Phi_1$ are globular, then their horizontal composition, and every 2-morphism globularily equivalent to it, is globular. We thus assume that one of $\Psi$ and $\Phi_1$ is non-globular. If $\Psi$ is globular, then the equation $t\Phi_1=s\Psi$ together with induction hypothesis implies that $\Phi_1$ is globular. An identical argument implies that if $\Phi_1$ is globular then $\Psi$ is globular. We conclude that both $\Psi$ and $\Phi_1$ are non-globular and thus by induction hypothesis are horizontal endomorphisms. This and equation $t\Phi_1=s\Psi$ implies that $\Psi\ast\Phi_1$ and thus $\Phi$ is a horizontal endomorphism. Assume now that $\Phi$ is a general morphism of $V^C_n$. Write $\Phi$ as a vertical composition $\Phi_k\circ...\circ\Phi_1$ where $\Phi$ is an element of $H^C_n$ for every $k$. Moreover, assume again that the length $k$ of this decomposition is minimal. An induction argument over $k$ together with an argument analogous to that presented in the base of the induction proves that $\Phi$ is a horizontal endomorphism. This concludes the proof.
\end{proof}

\noindent A direct consequence of proposition 4.4 is that neither double categories of the form \textbf{Cob}$(n)$ nor double categories \textbf{Alg} or $[W^*]^f$, presented in section 2 are globularily generated. We interpret this by saying that the condition of a double category being globularily generated is not trivial. We will compute the globularily generated piece of these double categories in section 6. This will provide non-trivial examples of globularily generated double categories. The following corollary follows immediately from the previous proposition.

\begin{cor}
Let $C$ be a globularily generated double category. Let $\Phi$ and $\Psi$ be 2-morphisms in $C$. Suppose $\Phi$ and $\Psi$ are composable. In that case horizontal composition $\Psi\ast\Phi$ is globular if and only if $\Phi$ and $\Psi$ are both globular.
\end{cor}

\noindent We conclude this section with the following technical lemma.

\begin{lem}
Let $C$ be a globularily generated double category. Let $\Phi$ be a 2-morphism in $C$. If the vertical length of $\Phi$ is equal to 1 then $\Phi$ can be written as a vertical composition of the form 

\[\Psi_k\circ\Phi_k\circ...\circ\Psi_1\circ\Phi_1\circ\Psi_0\]

\noindent where $\Phi_i$ is a horizontal identity for every $1\leq i\leq k$ and $\Psi_i$ is globular for every $0\leq i\leq k$.
\end{lem}

\begin{proof}
Let $C$ be a globularily generated double category. Let $\Phi$ be a 2-morphism in $C$. Suppose that the vertical length of $\Phi$ is equal to 1. We wish to prove, in this case, that $\Phi$ admits a decomposition as described in the statement of the lemma.

Suppose first that $\Phi$ is an element of $H^C_1$. In that case $\Phi$ is either globular or $\Phi$ is the horizontal identity of a vertical morphism in $C$. Suppose first that $\Phi$ is globular. In that case make $k=0$ and $\Psi_0=\Phi$. Suppose now that $\Phi$ is the horizontal identity of a vertical morphism $\alpha$ in $C$, with domain and codomain $x$ and $y$ respectively. In that case make $k=1$, make $\Psi_0$ to equal to the identity 2-morphism of the horizontal identity of $x$, make $\Phi_1$ to be equal to $\Phi$ and make $\Psi_1$ to be equal to the identity 2-morphism of horizontal identity of $y$.

Suppose now that $\Phi$ is a general morphism of the first vertical category $V_1^C$ associated to $C$. Write $\Phi$ as the vertical composition $\Phi=\Theta_m\circ...\circ\Theta_1$, where $\Theta_i$ is an element of $H^C_1$ for every $i$. Choose this decomposition in such a way that its length $m$ is minimal. We proceed by induction on $m$. In the case in which $m$ is equal to 1 $\Phi$ is an element of $H^C_1$. Suppose now that $m$ is strictly greater than 1 and that the result is true for every 2-morphism in $V^C_1$ that can be written as a vertical composition of strictly less than $m$ elements of $H^C_1$. Write $\Psi$ for vertical composition $\Theta_{m-1}\circ...\circ\Theta_1$. In that case $\Psi$ admits a decomposition as

\[\Psi=\Psi_k\circ\Phi_k\circ...\circ\Psi_1\circ\Phi_1\circ\Psi_0\]

\noindent for some $k$, where $\Phi_i$ is a horizontal identity for every $1\leq i\leq k$ and $\Psi_i$ is globular for every $0\leq i\leq k$. Since $\Theta_m$ is an element of $H^C_1$ then it is either globular or it is the horizontal identity of a vertical morphism in $C$. Suppose first that $\Theta_m$ is globular. In that case write $\Psi_k'$ for vertical composition $\Theta_m\circ\Psi_k$. In that case the decomposition

\[\Phi=\Psi_k'\circ\Phi_k\circ...\circ\Psi_1\circ\Phi_1\circ\Psi_0\]

\noindent satisfies the conditions of the lemma. Suppose now that $\Theta_m$ is the vertical identity of a vertical morphism $\alpha$, with domain and codomain $x$ and $y$ respectively. In that case write $\Phi_{k+1}$ for $\Theta_m$ and write $\Psi_{k+1}$ for the identity 2-endomorphism of the identity horizontal endomorphism of $x$. In that case the decomposition

\[\Phi=\Psi_{k+1}\circ\Phi_{k+1}\circ...\circ\Psi_1\circ\Phi_1\circ\Psi_0\]

\noindent satisfies the conditions of the lemma. This concludes the proof.
\end{proof}

\section{Categorification}

\noindent In this section we present an extension of the definition of the vertical filtration of a globularily generated double category presented in the previous section, to a filtration on the globularily generated piece funtor. We begin with the following lemma.

\

\begin{lem}
Let $C$ and $D$ be globularily generated double categories. Let $F:C\to D$ be a double functor from $C$ to $D$. Let $n$ be a positive integer. The image of the $n$-th vertical category $V_n^C$ associated to $C$, under morphism functor $F_1$ of $F$, is contained in the $n$-th vertical category $V_n^{D}$ associated to $D$. Moreover, in the case in which $C,D$, and $F$ are strict, the image of the $n$-th horizontal category $H_n^C$ associated to $C$, under morphism functor $F_1$ of $F$, is contained in the $n$-th horizontal category $H_n^{D}$ associated to $D$.
\end{lem}

\begin{proof}
Let $C$ and $D$ be globularily generated double categories. Let $F:C\to D$ be a double functor from $C$ to $D$. Let $n$ be a positive integer. We wish to prove that the image of the $n$-th vertical category $V_n^C$ associated to $C$, under morphism functor $F_1$ of $F$, is a subcategory of the $n$-th vertical category $V_n^{D}$ associated to $D$. Moreover, we wish to prove that if $C,D$, and $F$ are all strict then the image of the $n$-th horizontal category $H_n^C$ associated to $C$, under morphism functor $F_1$ of $F$, is a subcategory of the $n$-th horizontal category $H_n^{D}$ associated to $D$.

We proceed by induction on $n$. Let $\Phi$ be a 2-morphism in the first vertical category $V_1^C$ associated to $C$. We wish to prove, in this case that $F_1\Phi$ is a morphism in the first vertical category $V_1^{D}$ associated to $D$. Suppose first that $\Phi$ is an element of $H_1^C$. In that case $\Phi$ is either globular or the horizontal identity of a vertical morphism in $C$. Suppose first that $\Phi$ is the horizontal identity of a vertical morphism $\alpha$ in $C$. In that case the image $F_1\Phi$ of $\Phi$ under functor $F_1$ is globularily conjugate to the horizontal identity of the image $F_0\alpha$ of $\alpha$ under functor $F_0$, and is thus a morphism in category $V_1^{D}$. Observe that in the case in which double functor $F$ is strict $F_1\Phi$ is precisely the horizontal identity of vertical morphism $F_0\alpha$ and is thus an element of $H^D_0$. From this and from the fact that collection of globular 2-morphisms of a double category is invariant under the application of double functors it follows that the image of collection $H_1^C$, under morphism functor $F_1$, is contained in collection of morphisms of first vertical category $V_1^{D}$ of $D$. Moreover, in the case in which $F$ is strict, the image of $H_1^C$ under $F_1$ is contained in $H_1^{D}$. Suppose now that $\Phi$ is a general element of the first vertical category $V_1^C$ associated to $C$. Write $\Phi$ as a vertical composition

\[\Phi=\Phi_k\circ...\circ\Phi_1\]

\noindent where $\Phi_i$ is an element of $H_1^C$ for every $1\leq i\leq k$. In that case the image of $\Phi$ under morphism functor $F_1$ of $F$ is equal to vertical composition 

\[F_1\Phi_k\circ...\circ F_1\Phi_1\]

\noindent which is a morphism of the first vertical category $V_1^{D}$ associated to $D$. Thus the image of the first vertical category $V_1^C$ associated to $C$, under morphism functor $F_1$ of $F$ is a subcategory of the first vertical category $V_1^{D}$ associated to $D$. Moreover, if we assume that $C,D$, and $F$ are strict, $H_1^{C}$ and $H_1^{D}$ are categories, and the image of $H_1^C$ under morphism functor $F_1$ of $F$ is a subcategory of $H_1^{D}$.

Let $n$ now be strictly greater than 1. Suppose that the conclusions of the proposition are true for every $m<n$. Let $\Phi$ now be a morphism in the $n$-th vertical category $V_n^C$ associated to $C$. We wish to prove in this case that the image $F_1\Phi$ of $\Phi$ under morphism functor $F_1$ of $F$ is a morphism in the $n$-th vertical category $V_n^{D}$ associated to $D$. Suppose first that $\Phi$ is a morphism in $H_n^C$. Write $\Phi$, up to globular equivalences, as a horizontal composition of the form

\[\Phi\equiv\Phi_k\ast...\ast\Phi_1\]

\noindent where $\Phi_i$ is an element of the $n-1$-th vertical category $V^C_{n-1}$ associated to $C$ for every $1\leq i\leq k$. In that case the image $F_1\Phi$ under functor $F_1$ is globularily equivalent to any possible interpretation of horizontal composition

\[F_1\Phi_k\ast...\ast F_1\Phi_1\]

\noindent in $D$. By induction hypothesis $F_1\Phi_i$ is a morphism of the $n-1$-th vertical category $V_{n-1}^{D}$ associated to $D$ for every $1\leq i\leq k$ and thus any interpretation of the horizontal composition above is a morphism of the $n$-th horizontal cateogory $V_n^{D}$ associated to $D$. We conclude that the image $F_1\Phi$ of 2-morphism $\Phi$ under functor $F_1$ is a morphism in the $n$-th vertical category $V_n^{D}$ associated to $D$. Moreover, if $F$ is strict then image $F_1\Phi$ of $\Phi$ under functor $F_1$ is an element of $H_n^{D}$. Suppose now that $\Phi$ is a general morphism of the $n$-th vertical cateogry $V_n^C$ associated to $C$. Write $\Phi$ as a vertical composition of the form

\[\Phi=\Phi_k\circ...\circ\Phi_1\]

\noindent where $\Phi_i$ is an element of $H^C_n$ for every $1\leq i\leq k$. In that case the image $F_1\Phi$ of $\Phi$ under morphism functor $F_1$ of $F$ is equal to vertical composition

\[F_1\Phi_k\circ...\circ F_i\Phi_1\]

\noindent in $D$ and thus is an element of $n$-th vertical category $V_n^{D}$ of $D$. We conclude that the image of the $n$-th vertical category $V_n^C$ associated to $C$, under morphism functor of double functor $F$, is a subcategory of the $n$-th vertical category $V_n^{D}$ associated to $D$ and that if $C,D$, and $F$ are strict then moreover the image of the $n$-th horizontal category $H_n^C$ associated to $C$, under morphism functor $F_1$ of $F$, is a subcategory of the $n$-th horizontal category $H_n^{D}$ associated to $D$. This concludes the proof.
\end{proof}

\noindent Let $C$ and $D$ be globularily generated double categories. Let $F:C\to D$ be a double functor from $C$ to $D$. Let $n$ be a positive integer. We write $V_n^F$ for restriction, to the $n$-th vertical category $V_n^C$ associated to $C$, of morphism functor $F_1$ of $F$. Thus defined $V_n^F$ is, by lemma 5.1, a functor from the $n$-th vertical category $V_n^C$ associated to $C$ to the $n$-th vertical category $V_n^{D}$ associated to $D$. We call $V_n^F$ the $n$\textbf{-th vertical functor associated to} $F$. 

The pair formed by the function associating the $n$-th vertical category $V_n^C$ associated to $C$ to every double category $C$ and the $n$-th vertical double functor $V_n^F$ to every double functor $F$ forms a functor from the underlying category of 2-category \textbf{gCat} of globularily generated double categories, double functors, and globularily generated double natural transformations to the underlying category of 2-cateogory \textbf{Cat} of categories, functors, and natural transformations. We denote this functor by $V_n$. We call $V_n$ the $n$-\textbf{th vertical functor}.

\

\noindent Denote by $\pi_0$ and $\pi_1$ the 2-functors from \textbf{dCat} to \textbf{Cat} such that for every double category $C$, $\pi_0C$ and $\pi_1C$ are equal to the object category $C_0$ of $C$ and to the morphism category $C_1$ of $C$ respectively, such that for every double functor $F$, $\pi_0F$ and $\pi_1F$ are equal to the object functor $F_0$ of $F$ and to the  morphism functor $F_1$ of $F$ respectively, and finally, such that for every double natural transformation $\eta$, $\pi_0\eta$ and $\pi_1\eta$ are equal to the object natural transformation $\eta_0$ associated to $\eta$ and to the morphism natural transformation $\eta_1$ associated to $\eta$ respectively. We call $\pi_0$ and $\pi_1$ the \textbf{object projection} and the \textbf{morphism projection} 2-functors of \textbf{dCat} respectively. We keep denoting by $\pi_0$ and $\pi_1$ restrictions of object and morphism projections of 2-category \textbf{dCat}, to sub 2-category \textbf{gCat}. 

\

\noindent Given a double category $C$ and positive integers $m$ and $n$ such that $n\geq m$, the $m$-th vertical category $V_m^{\gamma C}$ associated to the globularily generated piece $\gamma C$ of $C$ is a subcategory of the $n$-th vertical category $V_n^{\gamma C}$ associated to $\gamma C$. We write $\eta_{m,n}^\gamma C$ for the inclusion functor of category $V_m^{\gamma C}$ in $V_n^{\gamma C}$.

Observe that given a double functor $F:C\to D$ from a double category $C$ to a double category $D$, the fact that vertical functors $V_m^{\gamma F}$ and $V_n^{\gamma F}$ associated to double functor $\gamma F$ are restrictions of the morphism functor $\gamma F_1$ of $\gamma F$, implies that the square:

\begin{center}

\begin{tikzpicture}
\matrix(m)[matrix of math nodes, row sep=3.6em, column sep=6em,text height=1.5ex, text depth=0.25ex]
{V_m^{\gamma C}&V_m^{\gamma D}\\
V_n^{\gamma C}&V_n^{\gamma D}\\};
\path[->,font=\scriptsize,>=angle 90]
(m-1-1) edge node[auto] {$V_m^{\gamma F}$} (m-1-2)
        edge node[left] {$\eta_{m,n}^{\gamma C}$} (m-2-1)
(m-2-1) edge node[below]{$V_n^{\gamma F}$} (m-2-2)
(m-1-2) edge node[right] {$\eta_{m,n}^{\gamma D}$} (m-2-2);
\end{tikzpicture}

\end{center}

\noindent commutes. That is, if we denote by $\eta_{m,n}$ the collection of inclusions $\eta_{m,n}^{\gamma C}$ with $C$ running through collection of double categories, then $\eta_{m,n}$ is a natural transformation from the composition $V_m\gamma$ of the globularily generated piece 2-functor $\gamma$ and the $m$-th vertical functor $V_m$, to the composition $V_n\gamma$ of the globularily generated piece 2-functor $\gamma$ and the $n$-th vertical functor $V_n$. The sequence formed by functors $V_n\gamma$ together with the collection formed by natural transformations $\eta_{m,n}$ forms a diagram in the underlying category of 2-category \textbf{Cat} with base in the underlying category of 2-category \textbf{dCat}. 

\

\noindent We will write $\gamma_1$ for the composition $\pi_1\gamma$ of the globularily generated piece 2-functor $\gamma$ and the morphism projection $\pi_1$. The following proposition says that the limit of the diagram above is functor $\gamma_1$. Its proof follows directly from lemma 4.1.

\begin{prop}
Functor $\gamma_1$ defined above is equal to the limit $\varinjlim V_n\gamma$ of diagram formed by sequence of vertical functors $V_n\gamma$ and collection natural transformations $\eta_{m,n}$.
\end{prop}

\noindent Let now $C$ and $D$ be strict globularily generated double categories and let $F:C\to D$ be a strict double functor from $C$ to $D$. If $n$ is a positive integer, then from the assumption that $C,D$, and $F$ are strict, and from lemma 5.1, it follows that the pair formed by morphism function of object functor $F_0$ of $F$ and morphism function of morphism functor $F_1$ of $F$ restrict to a functor from the $n$-th horizontal category $H_n^C$ associated to $C$ to the $n$-th horizontal category $H_n^D$ associated to $D$. We denote this functor by $H_n^F$ and we call it the $n$-\textbf{th horizontal functor associated to} $F$. 

\

\noindent Denote by $\overline{\textbf{dCat}}$ the sub 2-category of \textbf{dCat} generated by collection of strict double categories and collection of strict double functors between them, and denote by $\overline{\mbox{\textbf{gCat}}}$ the sub 2-category of $\overline{\textbf{dCat}}$ generated by collection of strict globularily generated double categories. Given a positive integer $n$, the pair of functions associating, for every strict globularily generated double category $C$ the $n$-th horizontal category $H_n^C$ associated to $C$, and to every strict double functor $F$, the $n$-th horizontal functor $H_n^F$ associated to $F$, is a functor from the underlying category of 2-category $\overline{\mbox{\textbf{gCat}}}$, to the underlying category of 2-category \textbf{Cat} of categories, functors, and natural transformations.  

\

\noindent Given a strict double category $C$, we denoted, in section 4, by $\tau C$ the category whose collection of objects is collection of vertical morphisms of $C$ and whose collection of morphisms is collection of 2-morphisms of $C$. We called category $\tau C$ the \textbf{transversal category associated to} $C$. Given a strict double functor $F:C\to D$ from a strict double category $C$ to a strict double category $D$, we denote by $\tau F$ the functor from the transversal category $\tau C$ associated to $C$ to the transversal category $\tau D$ associated to $D$, such that object and morphism functions of $\tau F$ are the morphism function of object functor $F_0$ associated to $F$ and the morphism function of morphism functor $F_1$ associated to $F$ respectively. We call $\tau F$ the \textbf{transversal functor associated to} $F$. 

The pair of functions associating the transversal category $\tau C$ to a double category $C$ and the transversal functor $\tau F$ to a double functor $F$ forms a functor from the underlying category of 2-category $\overline{\mbox{\textbf{dCat}}}$ to \textbf{Cat}. We denote this functor by $\tau$. We call $\tau$ the \textbf{transversal category functor}. We keep denoting by $\tau$ the restriction of the transversal functor to the underlying category of 2-category $\overline{\mbox{\textbf{gCat}}}$.

\

\noindent Given a strict globularily generated double category $C$ and positive integers $m$ and $n$ such that $m\leq n$, the $m$-th horizontal category $H_m^C$ associated to $C$ is a subcategory of the $n$-th horizontal category $H_n^C$ associated to $C$. In this case denote by $\nu^C_{m,n}$ the inculsion functor of $H_m^C$ in $H_n^C$.

Given a strict double functor $F:C\to D$, from a strict double category $C$ to a strict double category $D$, for every pair of positive integers $m$ and $n$ such that $m\leq n$, the diagram

\begin{center}

\begin{tikzpicture}
\matrix(m)[matrix of math nodes, row sep=3.6em, column sep=6em,text height=1.5ex, text depth=0.25ex]
{H_m^{\gamma C}&H_m^{\gamma D}\\
H_n^{\gamma C}&H_n^{\gamma D}\\};
\path[->,font=\scriptsize,>=angle 90]
(m-1-1) edge node[auto] {$H_m^{\gamma F}$} (m-1-2)
        edge node[left] {$\nu_{m,n}^{\gamma C}$} (m-2-1)
(m-2-1) edge node[below]{$H_n^{\gamma F}$} (m-2-2)
(m-1-2) edge node[right] {$\nu_{m,n}^{\gamma D}$} (m-2-2);
\end{tikzpicture}

\end{center}

\noindent commutes. That is, if in this case we denote by $\nu_{m,n}$ collection formed by inclusions $\nu_{m,n}^{\gamma C}$ with $C$ running through the collection of strict double categories, then $\nu_{m,n}$ is a natural transformation from the composition $H_m\gamma$ of the underlying functor of the globularily generated piece 2-functor $\gamma$ and the $m$-th horizontal functor, to the composition $H_n\gamma$ of the globularily generated piece functor $\gamma$ and the $n$-th horizontal category functor $H_n$. The sequence formed by functors $H_n$ together with the collection of natural transformations $\nu_{m,n}$ forms a diagram in the underlying category of \textbf{Cat}, with base in the underlying category of $\overline{\textbf{dCat}}$. 
\

\noindent We will denote the composition $\tau\gamma$ of the underlying functor of the globularily generated piece 2-functor $\gamma$ and the transversal functor $\tau$ by $\gamma^{\tau}$.  
The following proposition says that the limit of the above diagram is functor $\gamma^{\tau}$. Its proof now follows directly from lemma 4.2.

\begin{prop}
Functor $\gamma^\tau$ defined above is the limit $\varinjlim  H_n\gamma$ formed by sequence of horizontal functors $ H_n\gamma$ and collection of natural transformations $\nu_{m,n}$.
\end{prop}

\section{Computations}

\noindent In this final section we present explicit computations of the globularily generated piece of double categories introduced in section 2. We begin by computating the globularily generated piece $\gamma$\textbf{Cob}$(n)$ of double category \textbf{Cob}$(n)$ of $n$-dimensional manifolds, diffeomorphisms, cobordisms, and equivariant diffeomorphisms, for every positive integer $n$

\

\noindent We will write a horizontal equivariant endomorphism $(f,\Phi,f)$ in double category \textbf{Cob}$(n)$ simply as $(f,\Phi)$. If an equivariant morphism in \textbf{Cob}$(n)$ is written in this way it will be assumed it is a horizontal endomorphism. We will say that cobordisms $M$ and $N$ from a closed manifold $X$ to itself are \textbf{globularily diffeomorphic} if $M$ and $N$ are diffeomorphic relative to $X$. 

\

\noindent In order to explicitly compute globularily generated piece $\gamma$\textbf{Cob}$(n)$ of double category \textbf{Cob}$(n)$, by proposition 4.4 we need only to compute collection of non-globular, globularily generated 2-morphisms between horizontal endomorphisms in \textbf{Cob}$(n)$. We begin with the following lemma.

\begin{lem}
Let $n$ be a positive integer. Let $X$ and $Y$ be closed $n$-dimensional manifolds. Let $M$ be a cobordism from $X$ to $X$ and let $N$ be a cobordism from $Y$ to $Y$. If there exist non-globular globularily generated diffeomorphisms from $M$ to $N$ then $M$ and $N$ are globularily diffeomorphic to idenitity cobordisms $i_X$ and $i_Y$ respectively.
\end{lem}

\begin{proof}
Let $n$ be a positive integer. Let $X$ and $Y$ be closed $n$-dimensional manifolds. Let $M$ be a cobordism from $X$ to $X$ and let $N$ be a cobordism from $Y$ to $Y$. Suppose there exists a non-globular globularily generated diffeomorphism from $M$ to $N$. In that case we wish to prove that $M$ and $N$ are globularily diffeomorphic to identity cobordisms $i_X$ and $i_Y$ respectively.

Let $(f,\Phi):M\to N$ be a non-globular globularily generated diffeomorphism from $M$ to $N$. We proceed by induction on the vertical length of $(f,\Phi)$ to prove that the existence of $(f,\Phi)$ implies that $M$ and $N$ are globularily diffeomorphic to horizontal identities $i_X$ and $i_Y$ respectively. Suppose first that the vertical length of $(f,\Phi)$ is equal to 1. By lemma 4.6 there exists a decomposition of $(f,\Phi)$ as a vertical composition of the form

\[(id_{X_k},\Psi_k)\circ(f_k,f_k\times id_{[0,1]})\circ...\circ(id_{X_1},\Psi_1)\circ(f_1,f_1\times id_{[0,1]})\circ(id_{X_0},\Psi_0)\]

\noindent where $X_0,...,X_k$ are $n$-dimensional manifolds, $X_0$ and $X_k$ are equal to $X$ and $Y$ respectively, $f_i:X_i\to X_{i+1}$ is a diffeomorphism from $X_i$ to $X_{i+1}$ for all $i\leq k-1$, and where $\Psi_i$ is a globular diffeomoerphism from $X_i$ to $X_i$ for all $i$. Since we assume that $(f,\Phi)$ is not globular then the length $k$ of this decomposition is greater than or equal to 1. The domain of the horizontal identity of $\Phi_1$ is equal to the horizontal identity $i_X$ of manifold $X$ and the codomain of the horizontal identity $\Phi_k$ is equal to the horizontal identity $i_Y$ of manifold $Y$. Thus $\Psi_0$ and $\Psi_k$ define globular diffeomorphisms between $M$ and $N$ and horizontal identities $i_X$ and $i_Y$ respectively.

Let $m$ be a positive integer strictly greater than 1. Assume now that the result is true for every pair of cobordisms admitting a non-globular globularily generated diffeomorphism of vertical length strictly less than $m$. Assume first that non-globular globularily generated diffeomorphism $(f,\Phi)$ is an element of $H^{\mbox{\textbf{Cob}}(n)}_m$. Write, in this case $(f,\Phi)$ as a horizontal composition

\[(f,\Phi)\equiv(f,\Phi_k)\ast...\ast(f,\Phi_1)\]

\noindent where $(f,\Phi_i)$ is a morphism in the $m-1$-th vertical category $V^{\mbox{\textbf{Cob}}(n)}_{m-1}$ associated to \textbf{Cob}$(n)$ for every $i\leq k$. Moreover, assume that the length $k$ of this decomposition is minimal. We proceed by induction on $k$ to prove that in this case the existence of $(f,\Phi)$ implies that $M$ and $N$ satisfy the conditions of the lemma. If $k=1$ then $(f,\Phi)$ is an element of the $m-1$-th vertical category $V^{\mbox{\textbf{Cob}}(n)}_{m-1}$ associated to \textbf{Cob}$(n)$ and by induction hypothesis its existence implies that $M$ and $N$ satisfy the conditions of the lemma. Suppose now that $k$ is strictly greater than 1. Write $(f,\Psi)$ for any representative of $(f,\Phi_k)\ast...\ast(f,\Phi_2)$. In that case the horizontal composition $(f,\Psi)\ast (f,\Phi_1)$ is equivalent to $(f,\Phi)$. From the assumption that $(f,\Phi)$ is not globular and from corollary 4.5 it follows that non of its conjugate morphisms is globular. Thus the horizontal composition $(f,\Psi)\ast(f,\Phi_1)$ is not globular and thus, again by corollary 3.6 neither $(f,\Psi)$ nor $(f,\Phi_1)$ is globular. Both $(f,\Psi)$ and $(f,\Phi_1)$ are globularily equivalent to the horizontal composition of strictly less than $k$ morphisms in the $m-1$-th vertical category $V^{\mbox{\textbf{Cob}}(n)}_{m-1}$ associated to \textbf{Cob}$(n)$. Let $M_1$ and $N_1$ be the domain and codomain of $(f,\Phi_1)$ and let $M_2$ and $N_2$ be the domain and codomain of $(f,\Psi_1)$. By induction hypothesis $M_1$ and $M_2$ are both globularily diffeomorphic to the horizontal identity $i_X$ of $X$ and both $N_1$ and $N_2$ are globularily diffeomorphic to the horizontal identity $i_Y$ of $Y$. It follows that $M_2\ast M_1$ is globularily diffeomorphic to the horizontal identity $i_X$ of $X$ and that $N_2\ast N_1$ is globularily diffeomorphic to the horizontal identity $i_Y$ of $Y$. Finally, by the exchange property in \textbf{Cob}$(n)$ we conclude that $M$ and $N$ are globularily diffeomorphic to horizontal identities $i_X$ and $i_Y$ of $X$ and $Y$ respectively.

Suppose now that $(f,\Phi)$ is a general element of the $m$-th vertical category $V^{\mbox{\textbf{Cob}}(n)}_m$ associated to \textbf{Cob}$(n)$. In that case write $(f,\Phi)$ as a vertical composition

\[(f,\Phi)=(f_k,\Phi_k)\circ...\circ(f_1,\Phi_1)\]

\noindent where $(f_i,\Phi_i)$ is an element of $H_m^{\mbox{\textbf{Cob}}(n)}$ for every $i$. Moreover, assume that the length $k$ of this decomposition is minimal. We again proceed by induction on $k$. If $k=1$ then $(f,\Phi)$ is an element of $H^{\mbox{\textbf{Cob}}(n)}_m$. Suppose now that $k$ is strictly greater than 1 and that the existence of a non-globular globularily generated diffeomorphism in the $m$-th vertical category $V^{\mbox{\textbf{Cob}}(n)}_m$ associated to \textbf{Cob}$(n)$, between manifolds $X$ and $Y$, that can be written as a vertical composition of strictly less than $k$ diffeomorphisms in $H^{\mbox{\textbf{Cob}}(n)}_m$ implies the conclusion of the lemma for $X$ and $Y$. Write $(g,\Psi)$ for composition $(f_k,\Phi_k)\circ...\circ(f_2,\Phi_2)$. In that case $(f,\Phi)$ is equal to vertical composition $(g,\Psi)\circ(f_1,\Phi_1)$. Moreover, from the assumption that $(f,\Phi)$ is not globular it follows that one of $(g,\Psi)$ or $(f_1,\Phi_1)$ is non-globular. Assume first that $(g,\Psi)$ is globular. In that case source and taget of $(f_1,\Phi_1)$ are both equal to $f$. By induction hypothesis the domain and codomain of $(f,\Phi_1)$ are globularily diffeomorphic to the horizontal identity $i_X$ of $X$ and the horizontal identity $i_Y$ of $Y$ respectively. The domain of $(f,\Phi)$ is equal to the codomain of $(f,\Phi_1)$ and $(g,\Psi)$ defines a globular diffeomorphism between the codomain of $(f,\Phi)$ and the codomain of $(f_1,\Phi_1)$. We conclude that in this case, the existence of non-globular globularily generated diffeomorphism $(f,\Phi)$ implies the existence of a globular diffeomorphism between $M$ and the horizontal identity $i_X$ of $X$ and between $N$ and the horizontal identity $i_Y$ of $Y$. The case in which it is assumed that $(f_1,\Phi_1)$ is globular is handled analogously. Suppose now that neither $(g,\Psi)$ nor $(f_1,\Phi_1)$ are globular. In that case, the induction hypothesis implies that there exists a globular diffeomorphism between $M$, which is the domain of $(f_1,\Phi_1)$, and the horizontal identity $i_X$ of $X$ and that there exists a globular diffeomorphism between $N$, which is the codomain of $(g,\Psi)$, and the horizontal identity $i_Y$ of $Y$. This concludes the proof.  
\end{proof}

\noindent As a consequence of lemma 6.1, in order to compute the globularily generated piece $\gamma$\textbf{Cob}$(n)$ of double category \textbf{Cob}$(n)$ it is enough to compute the collection of non-globular globularily generated diffeomorphisms between horizontal endomorphisms globularily diffeomorphic to horizontal identities of closed $n$-dimensional manifolds. This is achieved in the following proposition.

\begin{prop}
Let $n$ be a positive integer. Let $X$ and $Y$ be closed $n$-dimensional manifolds. Let $M$ be a cobordism from $X$ to $X$ and let $N$ be a cobordism from $Y$ to $Y$. Suppose that $N$ is globularily diffeomorphic to the identity cobordism $i_X$ associated to $X$ and that $N$ is globularily diffeomorphic to the identity cobordism $i_Y$ associated to $Y$. In that case every horizontal endomorphism from $M$ to $N$, in double category \textbf{Cob}$(n)$, is globularily generated and has vertical length equal to 1.
\end{prop}

\begin{proof}
Let $n$ be a positive integer. Let $X$ and $Y$ be closed $n$-dimensional manifolds. Let $M$ be a cobordism from $X$ to $X$, globularily diffeomorphic to the horizontal identity $i_X$ associated to manifold $X$ and let $N$ be a cobordism from $Y$ to $Y$, globularily diffeomorphic to the horizontal identity $i_Y$ associated to $Y$. We wish to prove, in this case, that every 2-morphism, in double category \textbf{Cob}$(n)$, from $M$ to $N$, is globularily generated and has vertical length equal to 1.

We first prove the proposition for the case in which $M$ and $N$ are equal to the horizontal identity cobordisms $i_X$ and $i_Y$ respectively. Let $(f,\Phi):i_X\to i_Y$ be a 2-morphism, in \textbf{Cob}$(n)$, from $i_X$ to $i_Y$. In that case the equivariant morphism $(id_X,(f^{-1}\times id_{[0,1]})\Phi)$ is a globular endomorphism of the horizontal identity $i_X$ of $X$ making the following triangle

\begin{center}

\begin{tikzpicture}
\matrix(m)[matrix of math nodes, row sep=3.6em, column sep=6em,text height=1.5ex, text depth=0.25ex]
{i_X&i_Y\\
i_X\\};
\path[->,font=\scriptsize,>=angle 90]
(m-1-1) edge node[auto] {$(f,\Phi)$} (m-1-2)
        edge node[left] {$(id_X,(f^{-1}\times id_{[0,1]})\Phi)$} (m-2-1)
(m-2-1) edge node[right] {$(f,f\times id_{[0,1]})$} (m-1-2);
\end{tikzpicture}

\end{center}

\noindent commute. Since $(id_X,(f^{-1}\times id_{[0,1]})\Phi)$ is globular and $(f,f\times id_{[0,1]})$ is the horizontal identity $i_f$ of diffeomorphism $f$ of $X$, we conclude that $(f,\Phi)$ is globularily generated and that its vertical length is equal to 1.

Suppose now that $M$ is a general cobordism from $X$ to $X$ globularily diffeomorphic to the horizontal identity $i_X$ of $X$ and that $N$ is a general cobordism from $Y$ to $Y$, globularily diffeomorphic to the horizontal identity $i_Y$ of $Y$. Let $(f,\Phi):M\to N$ be a general 2-morphism, in \textbf{Cob}$(n)$, from $M$ to $N$. Let $(id_X,\varphi):M\to i_X$ be a globular diffemorphism from $M$ to the horizontal identity $i_X$ of $X$ and let $(id_Y,\phi):N\to i_Y$ be a globular diffeomorphism from $N$ to the horizontal identity $i_Y$ of $Y$. In that case composition $(f,\Psi)=(id_Y,\phi)(f,\Phi)(id_X,\varphi^{-1})$ is a 2-morphism from the horizontal identity $i_X$ of $X$ to the horizontal identity $i_Y$ of $Y$ and is thus a morphism in the first vertical category $V^{\mbox{\textbf{Cob}}(n)}_1$ associated to \textbf{Cob}$(n)$. We conclude that $(f,\Phi)=(id_Y,\phi^{-1})(f,\Psi)(id_X,\varphi)$ is also a morphism in the first vertical category $V^{\mbox{\textbf{Cob}}(n)}_1$ associated to \textbf{Cob}$(n)$. This concludes the proof.
\end{proof}

\noindent By lemma 4.4, proposition 6.2 provides an explicit description of the globularily generated piece of double categories of the form \textbf{Cob}$(n)$. We now compute, using a procedure analogous to the one used to compute the globularily generated piece $\gamma$\textbf{Cob}$(n)$ of double categories of the form \textbf{Cob}$(n)$, the globularily generated piece $\gamma$\textbf{Alg} of double category \textbf{Alg} of complex algebras, unital algebra morphisms, bimodules, and equivariant bimodule morphisms. We make the same considerations regarding horizontal equivariant endomorphisms in \textbf{Alg} as we did with horizontal equivariant endomorphisms in \textbf{Cob}$(n)$. 

\

\noindent Given algebras $A$ and $B$, a left-right $A$-bimodule $M$ and a left-right $B$-bimodule $N$, we say that an equivariant morphism $(f,\varphi):M\to N$ from $M$ to $N$, is \textbf{2-subcyclic} if there exists a cyclic $A$-submodule $L$ of $N$, considering $N$ as an $A$-bimodule via $f$, and a cyclic $B$-submodule $K$ of $N$ such that inclusions Im$\varphi\subseteq L\subseteq K$ hold. Horizontal identities of algebra morphisms are examples of 2-subcyclic equivariant morphisms. Pair $(i,i^2)$ formed by inclusion of $\mathbb{Z}$ in $\mathbb{Q}$ and inclusion of $\mathbb{Z}^2$ in $\mathbb{Q}^2$ is an example of an non-2-subcyclic equivariant morphism from $\mathbb{Z}$-bimodule $\mathbb{Z}^2$ to $\mathbb{Q}$-bimodule $\mathbb{Q}^2$. 

\

\noindent In order to explicitly compute the globularily generated piece $\gamma$\textbf{Alg} of double category \textbf{Alg}, by proposition 4.4, we again need only to compute collection of non-globular, globularily generated 2-morphisms between horizontal endomorphisms in \textbf{Alg}. We begin with the following lemma.

\

\begin{lem}
Let $A$ and $B$ be algebras. Let $M$ and $M'$ be left-right $A$-bimodules and let $N$ and $N'$ be left-right $B$-bimodules. Let $(f,\varphi):M\to N$ be an equivariant morphism from $M$ to $N$ and let $(f,\varphi'):M'\to N'$ be equivariant morphism from $M'$ to $N'$. If both $(f,\varphi)$ and $(f,\varphi')$ are 2-subcyclic then relative tensor product $(f,\varphi\otimes_f\varphi')$ is 2-subcyclic.
\end{lem}

\begin{proof}
Let $A$ and $B$ be algebras. Let $M$ and $M'$ be left-right $A$-bimodules and let $N$ and $N'$ be left-right $B$-bimodules. Let $(f,\varphi):M\to N$ and $(f,\varphi'):M'\to N'$ be quivariant morphisms from $M$ to $N$ and from $M'$ to $N'$ respectively. Suppose both $(f,\varphi)$ and $(f,\varphi')$ are 2-subcyclic. We wish to prove in this case that the relative tensor product $(f,\varphi\otimes_f\varphi')$ is 2-subcyclic.

Let $L,L'$ and $K,K'$ be bimodules such that $L$ and $L'$ are $A$-cyclic submodules of $N$ and $N'$ respectively and such that $K$ and $K'$ are $B$-cyclic submodules of $N$ and $N'$ respectively. Moreover, let $L,L'$ and $K,K'$ satisfy inclusions Im$\varphi\subseteq L\subseteq K$ and Im$\varphi'\subseteq L'\subseteq K'$ respectively. The relative tensor product $L\otimes_AL'$ is an $A$-cyclic submodule of $N\otimes_AN'$, the relative tensor product $K\otimes_BK'$ is a $B$-cyclic submodule of $N\otimes_B N'$, $L\otimes_AL'$ is contained in $K\otimes_B' K'$, and finally Im$\varphi\otimes_f\varphi'$ is contained in $L\otimes_AL'$. This concludes the proof.
\end{proof}

\begin{prop}
Let $A$ and $B$ be algebras. Let $M$ be left-right $A$-bimodule and let $N$ be a left-right $B$-bimodule. In that case collection of non-globular globularily generated equivariant morphisms from $M$ to $N$ is precisely the collection of non-globular 2-subcyclic equivariant morphisms from $M$ to $N$. Moreover, every globularily generated equivariant morphism from $M$ to $N$ has vertical length equal to 1.
\end{prop}

\begin{proof}
Let $A$ and $B$ be algebras. Let $M$ be a left-right $A$-bimodule and let $N$ be a left-right $B$-bimodule. We wish to prove that collection of non-globular globularily generated equivariant morphisms from $M$ to $N$ is precisely the collection of non-globular 2-subcyclic equivariant morphisms from $M$ to $N$. Moreover, we wish to prove that every globularily generated equivariant morphism from $M$ to $N$ has vertical length equal to 1.

We first prove that every non-globular 2-subcyclic equivariant morphism from $M$ to $N$ is globularily generated. Let $(f,\varphi):M\to N$ be non-globular and 2-subcyclic. Let $K$ be a $B$-cyclic submodule of $N$ and let $L$ be an $A$-cyclic submodule of $K$ such that inclusions Im$\varphi\subseteq L\subseteq K$ hold. Let $j$ denote the inclusion of $K$ in $N$. Let $\overline{\varphi}$ denote the codomain restriction of $\varphi$ to $K$. Thus defined $j$ is a globular and equivariant morphism $(f,\overline{\varphi})$ makes the following triangle:

\begin{center}

\begin{tikzpicture}
\matrix(m)[matrix of math nodes, row sep=3.6em, column sep=6em,text height=1.5ex, text depth=0.25ex]
{M&N\\
K\\};
\path[->,font=\scriptsize,>=angle 90]
(m-1-1) edge node[auto] {$(f,\varphi)$} (m-1-2)
        edge node[left] {$(f,\overline{\varphi})$} (m-2-1)
(m-2-1) edge node[right] {$(id_B,j)$} (m-1-2);
\end{tikzpicture}

\end{center}

\noindent commute. Denote now by $j'$ the inclusion of $L$ in $K$ and denote by $\tilde{\varphi}$ the codomain restriction of $\varphi$ to $L$. Thus defined $j'$ is globular, and equivariant morphism $(f,\tilde{\varphi})$ makes the following triangle:

\begin{center}

\begin{tikzpicture}
\matrix(m)[matrix of math nodes, row sep=3.6em, column sep=6em,text height=1.5ex, text depth=0.25ex]
{M&K\\
L\\};
\path[->,font=\scriptsize,>=angle 90]
(m-1-1) edge node[auto] {$(f,\overline{\varphi})$} (m-1-2)
        edge node[left] {$(f,\tilde{\varphi})$} (m-2-1)
(m-2-1) edge node[right] {$(id_A,j')$} (m-1-2);
\end{tikzpicture}

\end{center}

\noindent commute. The following square:

\begin{center}

\begin{tikzpicture}
\matrix(m)[matrix of math nodes, row sep=3.6em, column sep=6em,text height=1.5ex, text depth=0.25ex]
{M&N\\
L&K\\};
\path[->,font=\scriptsize,>=angle 90]
(m-1-1) edge node[auto] {$(f,\varphi)$} (m-1-2)
        edge node[left] {$(f,\tilde{\varphi})$} (m-2-1)
(m-2-1) edge node[below] {$(id_A,j')$} (m-2-2)
(m-2-2) edge node[right] {$(id_B,j)$} (m-1-2);
\end{tikzpicture}

\end{center}

\noindent is thus commutative. Commutativity of this square is clearly equivalent to commutativity of square:

\begin{center}

\begin{tikzpicture}
\matrix(m)[matrix of math nodes, row sep=3.6em, column sep=6em,text height=1.5ex, text depth=0.25ex]
{M&N\\
L&K\\};
\path[->,font=\scriptsize,>=angle 90]
(m-1-1) edge node[auto] {$(f,\varphi)$} (m-1-2)
        edge node[left] {$(id_A,\tilde{\varphi})$} (m-2-1)
(m-2-1) edge node[below] {$(f,j')$} (m-2-2)
(m-2-2) edge node[right] {$(id_B,j)$} (m-1-2);
\end{tikzpicture}

\end{center}

\noindent Left and right hand sides of this last square are Globular. Finally, triangle:

\begin{center}

\begin{tikzpicture}
\matrix(m)[matrix of math nodes, row sep=3.6em, column sep=6em,text height=1.5ex, text depth=0.25ex]
{L&K\\
K\\};
\path[->,font=\scriptsize,>=angle 90]
(m-1-1) edge node[auto] {$(f,j')$} (m-1-2)
        edge node[left] {$(f,f)$} (m-2-1)
(m-2-1) edge node[right] {$(id_B,j')$} (m-1-2);
\end{tikzpicture}

\end{center}

\noindent commutes, which proves that equivariant morphism $(f,j')$ is a morphism in the first vertical category $V^{\mbox{\textbf{Alg}}}_1$ associated to \textbf{Alg}. We conclude that the 2-subcyclic equivariant morphism $(f,\varphi)$ is globularily generated and has vertical length equal to 1.

We now prove that every non-globular globularily generated equivariant morphism from $M$ to $N$ is 2-subcyclic. Let $(f,\varphi):M\to N$ be non-globular and globularily generated. Assume first that $(f,\varphi)$ is an element of $H^{\mbox{\textbf{Alg}}}_1$. From the assumption that $(f,\varphi)$ is non-globular it follows that $(f,\varphi)$ is the horizontal identity of an algebra morphism and thus is 2-subcyclic. Suppose now that $(f,\varphi)$ is a general morphism in the first vertical category $V^{\mbox{\textbf{Alg}}}_1$ associated to \textbf{Alg}. We wish to find, in this case, an $A$-cyclic submodule $L$ of $N$ and a $B$-cyclcic submodule $K$ of $N$ such that inclusions Im$\varphi\subseteq L\subseteq K$ hold. Write $(f,\varphi)$ as a vertical composition of the form:

\[(id_B,\psi_{k+1})\circ(f_k,\phi_k)\circ...\circ(f_1,\phi_1)\circ(id_A,\psi_1)\]

\noindent as in lemma 4.6 where $f_i$ is an algebra morphism for every $i\leq k$. Write $(f,\Phi)$ for composition $(f_k,\phi_k)\circ...\circ(f_1,\phi_1)$. Thus defined $(f,\Phi)$ is an equivariant morphism from left-right $A$-bimodule $_AA_A$ to left-right $B$-bimodule $_BB_B$. Now make $K$ to be equal to the image Im$\psi_1$ of $\psi_1$ and make $L$ to be equal to the image Im$\Phi\psi$ of composition $\Phi\psi$. Thus defined $K$ and $L$ satisfy the conditions required. We conclude that every equivariant morphism in the first vertical category $V^{\mbox{\textbf{Alg}}}_1$ associated to \textbf{Alg} is 2-subcyclic. From this and from lemma 6.3 it follows that every globularily generated equivariant morphism between $M$ and $N$ is 2-subcyclic. The fact that every non-globular globularily generated 2-morphism in \textbf{Alg} has vertical length equal to 1 follows from this and from the first part of the proof. This concludes the proof.   
\end{proof}

\noindent Proposition 6.4 provides an explicit description of the globularily generated piece $\gamma$\textbf{Alg} of double category \textbf{Alg}. A similar computation provides a complete description of the globularily generated piece $\gamma[W^*]^f$ of double cateogory $[W^*]^f$ of semisimple von Neumann algebras, finite algebra morphisms, bimodules, and equivariant bimodule morphisms. Examples of globularily generated double categories having 2-morphisms of vertical length strictly greater than 1 will be studied in subsequent papers. 

\section{Bibliography}

\

\noindent

\noindent [1] Michael Atiyah. New invariants of three and four dimensional manifolds, \textit{Proc. Symp. Pure Math. American Math. Soc. \textbf{48}: 285-299.}

\

\noindent [2] Michael Atiyah. Topological quantum field theories, \textit{Publications Mathematiques de l'IHES, \textbf{68}: 175-186.}

\

\noindent  [3] John Baez, James Dolan. Higher-dimensional algebra and topological quantum field theory. \textit{J. Math. Phys. \textbf{36} (1995).}

\

\noindent [4] Arthur Bartels, Christopher L. Douglas, Andre Henriques. Dualizability and index of subfactors \textit{Quantum topology \textbf{5}. 2014. 289-345.}

\

\noindent [5] Arthur Bartels, Christopher L. Douglas, Andre Henriques. Conformal nets I: Coordinate free nets. \textit{Int. Math. Res. Not. \textbf{13}(2015), 4975-5052.}

\

\noindent [6] Arthur Bartels, Christopher Douglas, Andre Henriques. Conformal nets II: Conformal blocks. \textit{Comm. Math. Phys. (2017), to appear.}

\

\noindent [7] Arthur Bartels, Christopher L. Douglas, Andre Henriques. Fusion of defects. \textit{Mem. Amer. Math. Soc. (2017), to appear.}

\

\noindent [8] Jean Benabou. Introduction to bicategories. \textit{LNM \textbf{47}, Springer 1967, 1-77.}

\

\noindent [9] Alain Connes. Noncommutative Geometry, \textit{Academic press 1994.}

\

\noindent [10] Charles Ehresmann. Categories structurees. \textit{Ann. Sci. Ecole Nor. Sup. (3), \textbf{80} 349-426, 1963.}

\

\noindent [11] Charles Ehresmann. Categories et structures: extraits. \textit{Seminaire Ehresmann. Topologie et geometrie differentielle \textbf{6} 1-3, 1964.}

\

\noindent [12] Uffe Haagerup. The standard form of von Neumann algebras, \textit{Math. Scand. \textbf{37} (1975), 271-283.}

\

\noindent [13] V.R.F. Jones. Index for subfactors, \textit{Invent. Math., \textbf{66}, 1983,pp. 1-25.}

\

\noindent [14] Max Kelly, Ross Street. Review of the elements of 2-categories. \textit{Category seminar, Proc. Sem. Sydney, 1972/1973, Lecture notes in math. \textbf{420} 75-103. Springer, Berlin, 1974.}

\

\noindent [15]  Hideki Kosaki. Extension of Jones' theory of index to arbitrary subfactors \textit{J. Func. Anal., \textbf{66}, 1986, pp. 123-140.}

\

\noindent [16] Hideki Kosaki, Roberto Longo. A remark on the minimal index of subfactors, \textit{J. of Func. Anal. \textbf{107}(2) 458-470.}

\

\noindent [17] N.P. Landsman. Bicategories of operator algebras and Poisson manifolds. \textit{Math. physics in math $\&$ physics (Siena, 2000), Fields Inst. Comm. \textbf{30}, 271-286, AMS 2001.}

\

\noindent [18] Tom Leinster. Higher Operads, Higher Categories. \textit{Cambridge University Press 2004.}

\

\noindent [19] Jacob Lurie. On the classification of topological field theories, \textit{Current developments in mathematics, 2008, 129-280, Int. Press, Sommerville, MA.}

\

\noindent [20] Saunders Mac Lane. Categories for the working mathematician, \textit{Volume 5, Graduate texts in mathematics. Springer, second edition, 1998.}

\

\noindent [21] Stephan Stolz, Peter Teichner. What is an elliptic object? \textit{Topology, Geometry and quantum field theory, London Math. Soc. LNS \textbf{308}, Cambridge Univ. Press 2004, 247-343.}

\

\noindent [22] Michael A. Schulman. Framed bicategories and monoidal fibrations, \textit{Theory and Applications of Categories \textbf{20}(2008), no. 18, 650-738.}

\

\noindent [23] Andreas Thom. A remark about the Connes fusion tensor product. \textsl{Theory Appl. Categ., \textbf{25}}

\end{document}